\documentclass[11pt]{article}

\usepackage[margin=1.05in]{geometry}
\usepackage{amsmath,amssymb,amsthm,mathtools}
\usepackage{microtype}
\usepackage{verbatim}
\usepackage[
  colorlinks=true,
  linkcolor=blue,
  citecolor=blue,
  urlcolor=blue,
  hypertexnames=false,
  pdftitle={Error Term Estimates for Sets of Remainders},
  pdfauthor={Omkar Baraskar, Prashant Gokhale, Sarvagya Jain, Adam Kie\.zun}
]{hyperref}

\newtheorem{theorem}{Theorem}[section]
\newtheorem{lemma}[theorem]{Lemma}
\newtheorem{proposition}[theorem]{Proposition}

\theoremstyle{remark}
\newtheorem{remark}[theorem]{Remark}
\newcommand{\seqnum}[1]{\href{https://oeis.org/#1}{\rm \underline{#1}}}
\newcommand{\C}{\mathbb C}

\newcommand{\1}{\mathbf 1}
\newcommand{\dd}{\,\mathrm d}
\newcommand{\e}{\mathrm e}
\newcommand{\Log}{\operatorname{Log}}
\newcommand{\Ein}{\operatorname{Ein}}

\title{Error Term Estimates for Sets of Remainders}
\author{%
  \begin{tabular}{@{}c@{}}
    \begin{tabular}{@{}c@{\hspace{0.12\textwidth}}c@{}}
      \begin{minipage}[t]{0.40\textwidth}
        \centering
        \textbf{Omkar Baraskar} \\
        {\small Cheriton School of Computer Science\\ University of Waterloo \\ \textit{ON N2L 3G1, Canada} \\
        \texttt{obaraska26@gmail.com}}
      \end{minipage}
      &
      \begin{minipage}[t]{0.40\textwidth}
        \centering
        \textbf{Prashant Gokhale} \\
        {\small University of Wisconsin-Madison \\
        \texttt{prashant.gokhale@wisc.edu}}
      \end{minipage}
    \end{tabular}
    \\[2cm]
    \begin{tabular}{@{}c@{\hspace{0.12\textwidth}}c@{}}
      \begin{minipage}[t]{0.40\textwidth}
        \centering
        \textbf{Sarvagya Jain} \\
        {\small Department of Mathematics and Statistics \\
        University of Turku \\
        \textit{20014 Turku, Finland} \\
        \texttt{sajain@utu.fi}}
      \end{minipage}
      &
      \begin{minipage}[t]{0.40\textwidth}
        \centering
        \textbf{Adam Kie\.zun} \\
        {\small Meta Superintelligence Labs} \\
        \texttt{akiezun@meta.com}
      \end{minipage}
    \end{tabular}
  \end{tabular}
}
\date{}

\begin{document}
\maketitle

\begin{abstract}
For a positive integer \(n\), let
\[
  S_1(n)=\{n\bmod k:1\leq k\leq\lfloor n/2\rfloor\},
  \qquad s_1(n)=|S_1(n)|.
\]
Baraskar and Vukusic proved
\(s_1(n)=\Xi n+O(n/(\log n\log\log n))\) for an explicit positive constant
\(\Xi\), and asked for a sharper error term.  Their computations suggested
that an \(O(n^{1/3})\) bound might hold.  We prove that, for every \(\alpha>\sqrt2\),
\[
  \limsup_{n\to\infty}
  \frac{s_1(n)-\Xi n}{n\exp\{-\alpha\sqrt{\log n\log\log n}\}}
  =\infty,
\]
and the corresponding limit inferior is \(-\infty\).  In particular, no
fixed power saving is possible.  We also prove the complementary upper
estimate
\[
  s_1(n)=\Xi n+O\left(
    n\exp\{-\alpha\sqrt{\log n\log\log n}\}
  \right)
\]
for every fixed \(\alpha\in (0,\sqrt{2})\), sharpening the previous
\(O(n/(\log n\log\log n))\) bound.  Thus, the constant \(\sqrt2\) is the threshold for bounds of
the form \(n\exp\{-\alpha\sqrt{\log n\log\log n}\}\).
\end{abstract}

\section{Introduction}
Let \(n\) be a positive integer. Baraskar and Vukusic in~\cite{BaraskarVukusic2026} studied the
remainder set
\[
S_1(n):=\{n\bmod k:1\leq k\leq \lfloor n/2\rfloor\}
\]
and denoted its cardinality by \(s_1(n)\). This sequence is listed as \seqnum{A283190} in the OEIS (On-Line Encyclopedia of Integer Sequences) \cite{oeis}.  For a prime \(p\), put
\begin{equation*}\label{eq:density-constant}
  V(p):=\prod_{\substack{q<p\\q\ \mathrm{prime}}}
    \left(1-\frac1q\right),
  \qquad
  \Xi:=\sum_{p\ \mathrm{prime}}\frac{V(p)}{p(p+1)}.
\end{equation*}
In~\cite[Theorem 1]{BaraskarVukusic2026}, Baraskar and Vukusic proved the asymptotic formula
\[
  s_1(n)=\Xi n+O\!\left(\frac{n}{\log n\log\log n}\right)
\]
and asked whether the error estimate could be improved
(see \cite[Problem~1]{BaraskarVukusic2026}).  Their
computations through \(10^7\) suggested that an error of order \(O(n^{1/3})\)
might be closer to the truth.  Our first result shows that one cannot have a power-saving error term asymptotically.

For \(\alpha>0\) and \(n\geq3\), put
\begin{equation*}\label{eq:R-alpha-def}
  R_\alpha(n)
  :=n\exp\{-\alpha\sqrt{\log n\log\log n}\}.
\end{equation*}

\begin{theorem}\label{thm:main}
For every real \(\alpha>\sqrt2\),
\begin{equation*}\label{eq:signed-omega-main}
  \limsup_{n\to\infty}\frac{s_1(n)-\Xi n}{R_\alpha(n)}=\infty,
  \qquad
  \liminf_{n\to\infty}\frac{s_1(n)-\Xi n}{R_\alpha(n)}=-\infty.
\end{equation*}
\end{theorem}

In particular, for every \(\alpha>\sqrt2\),
\begin{equation*}\label{eq:no-fixed-power}
  s_1(n)-\Xi n\ne O\!\left(R_\alpha(n)\right).
\end{equation*}

We also give a corresponding improvement to the error term of Baraskar and
Vukusic.  The proof, given in Section~\ref{sec:improved-error}, is an
application of the fundamental lemma of sieve theory.

\begin{theorem}\label{thm:improved-error}
For every fixed \(\alpha\in (0,\sqrt2)\),
\[
  s_1(n)=\Xi n+O\left(
    R_\alpha(n)
  \right).
\]
\end{theorem}
Together, Theorems~\ref{thm:main} and~\ref{thm:improved-error} show that, the constant \(\sqrt2\) is the
threshold for error bounds of the form
\[
  n\exp\{-\alpha\sqrt{\log n\log\log n}\}.
\]
It would be interesting to study what happens at $\alpha = \sqrt{2}$.

More generally, Baraskar and Vukusic
defined the iterated remainder sets
\[
S_0(n):=\{1,2,\ldots,\lfloor n/2\rfloor\},\qquad
S_{j+1}(n):=\{n\bmod k:k\in S_j(n)\setminus\{0\}\},
\]
and wrote
\[
s_j(n):=|S_j(n)|.
\]
In~\cite[Problem 2]{BaraskarVukusic2026}, Baraskar and Vukusic conjectured that for all $j\geq 2$, 
\[\lim_{n\to\infty}\frac{s_j(n)}{n}\]
does not exist. In a companion work~\cite{BaraskarGokhaleJainKiezun2026}, the authors show that the above conjecture is true.

The sets $S_j(n)$ are related to the problem
of estimating the length of Pierce expansions for rationals, studied by Erd\H{o}s and
Shallit~\cite{ErdosShallit1991} and, more recently, by Chase and
Pandey~\cite{ChasePandey2022}. See~\cite{BaraskarGokhaleJainKiezun2026}  for a detailed discussion of the relation
between Pierce expansions and iterated remainder sets.

The oscillation phenomenon observed in Theorem~\ref{thm:main} belongs to a substantial
tradition of results showing that sieve-theoretic counting functions cannot
be approximated too uniformly by their expected main terms.  Particularly
relevant is the work of Friedlander, Granville, Hildebrand and
Maier~\cite{FGHM}, who studied, among other quantities, the sifting function
\[
   \Phi(y,z)
   :=\#\{n\leq y:p\mid n\Longrightarrow p\geq z\},
\]
which is closely related to the rough-number counts that occur naturally in
our analysis (see~\eqref{eq:Mp-rough-interval} below).  They proved quantitative oscillation theorems for such
functions by studying the associated Dirichlet series
\[
   F(z;s)=\prod_{p<z}(1-p^{-s})\zeta(s)
\]
and applying the saddle-point method to a Perron-type integral; in the relevant
range the dominant contribution comes from a pair of non-real conjugate
saddle points.  

The analysis in~\cite{FGHM} is itself closely connected with the
saddle-point treatment of integers with restricted prime factors developed
by Hildebrand and Tenenbaum~\cite{HildebrandTenenbaum}.  More recently,
Gorodetsky~\cite{GorodetskySmooth} obtained oscillation results for smooth
numbers through an approximate explicit formula in which the effects of
zeros of the Riemann zeta function and of prime powers are made explicit.

The argument used here has a different organization. We do not try to express rough-number counting function as a Perron-type integral directly, nor do we isolate an oscillatory contribution arising from a zero of the zeta function. Instead, we study an appropriate transform of the error term. A sufficiently strong hypothetical bound for the error term implies correspondingly strong bounds on the high-order derivatives of the associated transform at the boundary point \(w=1\). On the other hand, it will turn out that this transform admits an explicit prime-sum representation, giving us a second description of the same derivatives. A saddle-point analysis of this prime sum shows that some of these derivatives must be much larger than the bounds obtained under the hypothetical strong bound on the error term. Thus, whereas in~\cite{FGHM} the saddle point is used directly to exhibit oscillation in a sifting function, here it is used to force rapid growth of the boundary derivatives, and the oscillation of the original error term follows from the incompatibility of these two descriptions.

\subsection{Overview}\label{sec:overview}
We outline the proof of Theorem~\ref{thm:main}, suppressing several
convergence and boundary-regularity issues that are addressed in later
sections.  Put \(s_1(0)=0\) and \(b(n)=s_1(n)-s_1(n-1)\).  Then
\[
  s_1(n)=\sum_{m\leq n} b(m),
\]
so the summatory problem is naturally encoded by the Dirichlet series
\[
  B(w):=\sum_{n\geq1}\frac{b(n)}{n^w}.
\]
Since \(s_1(n)\) has a linear main term, the Dirichlet series encoding the error term is isolated by
factoring off the pole of \(\zeta(w)\).  M\"obius inversion gives
\(b(n)=\sum_{m|n}a(m)\), with \[a(m):=\sum_{p\ \mathrm{prime}}
    \sum_{\substack{d\ge1\\P^+(d)<p}}
    \mu(d)\bigl(\1_{m=pd}-\1_{m=(p+1)d}\bigr);\]  see
Section~\ref{sec:mobius-inversion} and Lemma~\ref{lem:exact-formula}
below.  Hence
\[
  B(w)=\zeta(w)D(w),
  \qquad
  D(w):=\sum_{n\geq1}\frac{a(n)}{n^w}.
\]

The next step is to express the high derivatives of \(D\) at \(w=1\) as a
prime sum.  Ignoring for the moment the justification of the right-hand
derivatives, which is supplied in Section~\ref{sec:prime-sum} below, one
has
\[
  D_k:=D^{(k)}(1)=\sum_p T_p^{(k)}(1),
  \qquad
  T_p(w):=\bigl(p^{-w}-(p+1)^{-w}\bigr)
  \prod_{q<p}(1-q^{-w}).
\]
Writing \(\mathcal L_p:=\log p\) and
\[
  G_p(z):=\frac{T_p(1+z/\mathcal L_p)}{T_p(1)},
\]
the chain rule gives
\[
  D_k=\sum_p T_p(1)\mathcal L_p^kG_p^{(k)}(0).
\]

For large primes \(p\), the normalized factors \(G_p\) are close to the
entire function
\[
  f(z)=\exp\{-z+\Ein(z)\},
  \qquad
  \Ein(z)=\sum_{m\geq1}\frac{(-1)^{m+1}z^m}{m\,m!}.
\]
This limiting function can be understood by looking at \(\log G_p(z)\).
As in \eqref{eq:G-factorization} below, \(G_p\) factors as a normalized
difference term times a normalized Euler product.  For the difference term,
the change of variables \(w:=1+z/\mathcal L_p\) gives
\[
  \frac{p^{-w}-(p+1)^{-w}}{p^{-1}-(p+1)^{-1}}
  =
  \e^{-z}\left(1+p(1-\e^{-\delta_p z})\right),
  \qquad
  \delta_p:=\frac{\log(1+1/p)}{\mathcal L_p}.
\]
Since \(p\delta_p=1/\mathcal L_p+O(1/(p\mathcal L_p))\), this factor has
logarithm \(-z+o(1)\).  For the Euler product, taking logarithms and
keeping the first-order terms in the prime factors gives
\[
  \sum_{q<p}
  \log\left(
    \frac{1-q^{-1}\e^{-z\log q/\mathcal L_p}}{1-q^{-1}}
  \right)
  =
  \sum_{q<p}\frac{1-\e^{-z\log q/\mathcal L_p}}q+o(1).
\]
The prime number theorem then turns this sum into a Riemann integral on
the logarithmic scale \(t=\log q/\log p\):
\[
  \sum_{q<p}\frac{1-\e^{-z\log q/\mathcal L_p}}q
  \approx
  \int_0^1\frac{1-\e^{-zt}}t\dd t
  =\Ein(z).
\]
Thus \(\log G_p(z)\approx -z+\Ein(z)\), which is the logarithm of \(f(z)\).
Consequently, for the primes that contribute the most to \(D_k\), one expects
\[
  G_p^{(k)}(0)\approx f^{(k)}(0),
\]
and therefore the behaviour of \(D_k\) is governed by the derivatives of
\(f\).  This approximation is made uniform in Lemma~\ref{lem:uniform-Gp}
and transferred to derivatives in Lemma~\ref{lem:derivative-approximation}
below.

The derivatives of $f$ at $0$ are estimated by Cauchy's formula,
\begin{equation}\label{eq: cauchy for higher derivates}
    \frac{f^{(k)}(0)}{k!}
  =
  \frac{1}{2\pi i}
  \int_{|z|=r}\frac{f(z)}{z^{k+1}}\dd z.
\end{equation}
Writing \(H(z)=-z+\Ein(z)\), the saddle equation for this integral is
\[
  zH'(z)=1-z-\e^{-z}=k.
\]
For large \(k\) the relevant saddle points form a conjugate pair
\(z_k,\overline{z_k}\), with
\[
  z_k=-\log k-i\pi+o(1).
\]
See Lemma~\ref{lem:saddle-points} below for the construction and
asymptotics of these saddle points.
Choosing \(r=|z_k|\), the main contribution to the  integral in~\eqref{eq: cauchy for higher derivates} comes
from neighbourhoods of these two points.  More precisely,
\[
  \frac{f^{(k)}(0)}{k!}
  =
  Q_k+\overline{Q_k}+o(|Q_k|)
  =
  2|Q_k|\bigl(\cos\theta_k+o(1)\bigr),
\]
where \(Q_k:=f(z_k)/(z_k^k\sqrt{2\pi (-z_k)(k+z_k)})\) and
\(\theta_k:=\arg Q_k\).  The phase \(\theta_k\) varies in a sufficiently controlled manner that we can choose two sequences \(k_j^+,k_j^-\to\infty\)
for which \((-1)^{k_j^\pm-1}\cos\theta_{k_j^\pm}\) tends respectively to
\(1\) and \(-1\).  These estimates are proved in
Section~\ref{sec:saddle-analysis} below; the final coefficient asymptotic
is Proposition~\ref{prop:f-saddle}, and the phase selection is
Lemma~\ref{lem:phase-variation}.

The passage from \(f^{(k)}(0)\) back to \(D_k\) uses the positivity and
asymptotic size of the weights \(T_p(1)\mathcal L_p^k\).  The weighted
prime number theorem and Mertens' product theorem give, as shown in
Lemma~\ref{lem:weighted-prime-moment} below,
\[
  \sum_p T_p(1)\mathcal L_p^k
  =
  (\e^{-\gamma}+o(1))\Gamma(k-1),
\]
while the contribution from primes in the range where we do not have the required approximation is negligible.  Combining this
with the uniform approximation \(G_p\approx f\) yields
\[
  D_k
  =
  \e^{-\gamma}\Gamma(k-1)2k!|Q_k|
  \bigl(\cos\theta_k+o(1)\bigr).
\]
In particular, the derivatives \(D_k\) have large positive and large
negative values along suitable subsequences, after multiplication by the
natural sign \((-1)^{k-1}\).  The asymptotic formula is
Proposition~\ref{prop:Dk-asymptotic}, and the resulting signed lower
bounds are stated in Proposition~\ref{prop:Dk-lower}, both proved in
Section~\ref{sec:lower-bound-Dk} below.

It remains to show that such large values are incompatible with a
one-sided error bound of size \(R_\alpha(n)\) when \(\alpha>\sqrt2\).
This is done in Section~\ref{sec:upper-bound-Dk} below.
Write
\[
  E(n):=s_1(n)-\Xi n,
  \qquad
  h(t):=\e^{-t}E(\lfloor\e^t\rfloor).
\]
A Mellin-transform identity gives
\[
  D_k=\int_0^\infty h(t)\mathcal{K}_k(t)\dd t,
\]
where \(\mathcal{K}_k\) is a polynomial of degree at most \(k-1\) whose
leading term is \(k(-t)^{k-1}\).  Hence
\((-1)^{k-1}\mathcal{K}_k(t)\) has a fixed positive sign once \(t\) is a
sufficiently large multiple of \(k\); see
Lemmas~\ref{lem:integral-representation} and~\ref{lem:kernel-sign} below.
If \(E(n)\leq C R_\alpha(n)\), or if \(E(n)\geq-C R_\alpha(n)\), then for
each \(0<\beta<\alpha\) this identity gives a corresponding one-sided bound for
\((-1)^{k-1}D_k\), governed by
\[
  J_{k-1}(\beta)
  :=
  \int_1^\infty t^{k-1}
  \exp\{-\beta\sqrt{t\log t}\}\dd t.
\]
A gamma-integral majorant shows that
\[
  J_{k-1}(\beta)
  \leq
  \e^{o(k)}
  \left(
    \frac{2}{\beta^2}
  \right)^k
  \left(
    \frac{k^2}{\e^2\log k}
  \right)^k.
\]
This is Lemma~\ref{lem:J-upper}, and it leads to
Proposition~\ref{prop:Dk-upper}.  Choosing \(\sqrt2<\beta<\alpha\), this
upper bound is smaller than the signed saddle-point lower bounds along the
subsequences above.  This contradiction proves both the limsup and the
liminf assertions in Theorem~\ref{thm:main}.

Theorem~\ref{thm:improved-error} is proved separately in
Section~\ref{sec:improved-error}.  There we return to the exact
least-prime decomposition
\[
  s_1(n)=1+\sum_{p<n}|M_p(n)|,
\]
apply the fundamental lemma of sieve theory in the range
\(2\log n<p\leq Y\), and take
\[
  Y=\exp(2\sqrt{\log n\log\log n}).
\]
The small-prime and large-prime ranges are then absorbed into the same
error bound; see Lemmas~\ref{lem:interval-fundamental},
\ref{lem:squarefree-smooth-bound}, and~\ref{lem:improved-large-prime-tail}
below.
\subsection{Acknowledgments}
P. Gokhale thanks his advisors, Misha Khodak and Sandeep Silwal, for their encouragement, for fostering an environment that supports AI-assisted mathematical research, and for providing access to ChatGPT for Academic Researchers. S. Jain would like to thank his advisor, Kaisa Matom{\"a}ki, for her guidance and encouragement and for providing him with access to a ChatGPT Business subscription. S. Jain was supported by the Research Council of Finland (grant numbers 346307, 364214, and 370133) and by the University of Turku Graduate School's Exactus fellowship while working on this project.
\subsection{AI Assistance}
The argument for the improved error term was developed with assistance from ChatGPT.  The underlying improvement is a natural consequence of
the fundamental lemma of sieve theory, but ChatGPT helped organize the details
of the proof.  For the sharpness result, the authors initially expected that
a power saving might be possible.  Work with Codex led the authors
to the conclusion that no fixed power saving can hold. The original output used a \textit{non-standard} complex-analytic approach (at least for an analytic number theorist) and established only that the error is not $O(n^{1-\varepsilon})$ for any $\varepsilon>0$. Subsequently, the authors streamlined and optimized the argument into the form presented
here with 
GPT-assisted revisions. It should be noted that the authors pursued the two-sided oscillation result using saddle-point analysis only because of the original Codex output. In addition, the present approach draws substantially on ideas from that AI-generated output at several crucial points. The Codex agent had access to Lean as a tool, which it used to verify parts of the argument it generated. We do not have a Lean formalization for the current version of the argument. The authors checked the AI-assisted material and take full
responsibility for the mathematical content of the paper.
\section{An Exact Formula from M\"obius Inversion}\label{sec:mobius-inversion}
In this section we introduce a Dirichlet series that isolates the error term
in the asymptotic formula for \(s_1(n)\).
Set \(P^+(1)=1\) and $P^-(1) = \infty$, and for \(d>1\) let \(P^+(d)\) (resp. $P^-(d)$) be the largest (resp. smallest) prime
factor of \(d\).  Let \(\mu\) be the M\"obius function, and let
\(\1_E\) denote the indicator of a condition \(E\).  Define
\begin{equation}\label{eq:a-def}
  a(m):=\sum_{p\ \mathrm{prime}}
    \sum_{\substack{d\ge1\\P^+(d)<p}}
    \mu(d)\bigl(\1_{m=pd}-\1_{m=(p+1)d}\bigr).
\end{equation}
\begin{lemma}\label{lem:exact-formula}
For every integer \(n\ge1\),
\begin{equation}\label{eq:exact-formula}
  s_1(n)=\sum_{m\le n}a(m)\left\lfloor\frac nm\right\rfloor.
\end{equation}
\end{lemma}

\begin{proof}

Note that $s_1(1) = a(1) = 0$. By \cite[Eq.~(4)]{BaraskarVukusic2026}, for \(n\geq 2\),
\begin{equation}\label{eq:least-prime-decomposition}
  s_1(n)=1+\sum_{p<n}|M_p(n)|,
\end{equation}
where
\[
  M_p(n):=
  \left\{
    m\in
    \left[
      n-\left\lfloor\frac{n-p}{p+1}\right\rfloor,
      n-1
    \right]:
    p\mid m
    \quad \text{and}\quad 
    P^-(m)\geq p
  \right\}.
\]
Thus, on writing an element of
\(M_p(n)\) as \(pa\), one obtains
\begin{equation}\label{eq:Mp-rough-interval}
  |M_p(n)|
  =\#\{a\in I_p\cap\mathbb Z:P^-(a)\geq p\},
  \qquad
  I_p:=\left(\frac n{p+1},\frac{n-1}{p}\right].
\end{equation}
M\"obius inversion gives
\[
  \1_{P^-(q)\geq p}
  =\sum_{\substack{d\mid q\\P^+(d)<p}}\mu(d).
\]
Thus,
\begin{equation*}\label{eq:local-mobius-count}
  |M_p(n)|=
  \sum_{\substack{d\ge1\\P^+(d)<p}}\mu(d)
  \left(
    \left\lfloor\frac{n-1}{pd}\right\rfloor
    -\left\lfloor\frac{n}{(p+1)d}\right\rfloor
  \right).
\end{equation*}

Before replacing the first floor, note that we may extend the prime sum from \(p<n\) to \(p\le n\).
Now using
\begin{align}\label{eqn: floor diff}
   \1_{v\mid n} 
  =\left\lfloor\frac nv\right\rfloor-\left\lfloor\frac{n-1}{v}\right\rfloor
  \qquad(v\ge1),
\end{align}
with \(v=pd\) and then summing
over \(p\) gives
\begin{align*}
  s_1(n)
  &=
  1+
  \sum_{\substack{p,d\geq1\\ p\ \mathrm{prime}\\P^+(d)<p}}
  \mu(d)
  \left(
    \left\lfloor\frac{n}{pd}\right\rfloor
    -\1_{pd\mid n}
    -\left\lfloor\frac{n}{(p+1)d}\right\rfloor
  \right)\\
  &=
  1+
  \sum_{\substack{p,d\geq1\\ p\ \mathrm{prime}\\P^+(d)<p}}
  \mu(d)
  \left(
    \left\lfloor\frac{n}{pd}\right\rfloor
    -\left\lfloor\frac{n}{(p+1)d}\right\rfloor
  \right)
  -C(n),
\end{align*}
where the correction term is
\[
  C(n):=
  \sum_{\substack{p,d\geq1\\ p\ \mathrm{prime}\\
                   P^+(d)<p\\pd\mid n}}
  \mu(d).
\]
If \(\mu(d)\ne0\), then \(v=pd\) is squarefree and \(p=P^+(v)\); conversely, every squarefree \(v\ge2\) occurs once. Since \(\mu(d)=-\mu(v)\), for $n>1$, \[ C(n)=-\sum_{\substack{v\mid n\\v\ge2}}\mu(v)=1. \]
Thus, we obtain
\begin{equation}\label{eq:post-correction}
  s_1(n)=
  \sum_{\substack{p,d\geq1\\ p\ \mathrm{prime}\\P^+(d)<p}}
  \mu(d)
  \left(
    \left\lfloor\frac{n}{pd}\right\rfloor
    -\left\lfloor\frac{n}{(p+1)d}\right\rfloor
  \right).
\end{equation}

Grouping the terms in \eqref{eq:post-correction} according to the
denominator of the floor function finishes the proof.
\end{proof}

Put \(b(n)=s_1(n)-s_1(n-1)\), with \(s_1(0)=0\).  Using \eqref{eq:exact-formula} and~\eqref{eqn: floor diff} gives
\begin{equation}\label{eq:dirichlet-convolution}
  b(n)=\sum_{m\mid n}a(m).
\end{equation}
For \(\Re w>1\), define
\[
  D(w)=\sum_{m\ge1}\frac{a(m)}{m^w},
  \qquad
  B(w)=\sum_{n\ge1}\frac{b(n)}{n^w}.
\]
Here and below, \(\zeta\) denotes the Riemann zeta function.
If \(\sigma=\Re w>1\), then \eqref{eq:a-def} gives
\begin{align}
  \sum_m|a(m)|m^{-\sigma}
  &\le \sum_p\bigl(p^{-\sigma}+(p+1)^{-\sigma}\bigr)
    \prod_{q<p}(1+q^{-\sigma}) \notag\\
  &\le 2\zeta(\sigma)\sum_p p^{-\sigma}<\infty.
  \label{eq:absolute-convergence}
\end{align}
Thus both series converge absolutely, and
\eqref{eq:dirichlet-convolution} gives
\begin{equation}\label{eq:dirichlet-product}
  B(w)=\zeta(w)D(w).
\end{equation}
\section{Prime-Sum Representation of \texorpdfstring{$D^{(k)}(w)$}{the Derivatives of D(w)}}\label{sec:prime-sum}
In this section we decompose \(D(w)\) into a prime sum and set up the
analytic framework needed to make sense of the higher derivatives of
\(D(w)\) at \(w=1\).
For each prime \(p\), set
\[
  \mathcal L_p:=\log p,
  \qquad
  d_p(w):=p^{-w}-(p+1)^{-w},
\]
and
\[
  V_p(w):=\prod_{q<p}(1-q^{-w}),
  \qquad
  T_p(w):=d_p(w)V_p(w).
\]
Note that
\[
  \sum_{\substack{d\geq 1\\ P^+(d)<p}}\frac{\mu(d)}{d^w}
  =
  \prod_{q<p}(1-q^{-w})
  =
  V_p(w).
\]
Hence, by the absolute convergence of \(D(w)\) in the half-plane
\(\Re w>1\) (see \eqref{eq:absolute-convergence}), we may rearrange its defining series to obtain
\begin{equation*}\label{eq:prime-decomposition}
  D(w)=\sum_p T_p(w).
\end{equation*}

We next collect the derivative estimates required below.

\begin{lemma}\label{lem:derivative-estimates}
There exists a fixed integer \(J\geq 0\) such that the following
statements hold.

For every integer \(j\geq 0\), every prime \(p\), and every \(w\) with
\(\Re w\geq 1\), one has
\begin{equation}\label{eq:Vp-derivative-bound}
  \bigl|V_p^{(j)}(w)\bigr|
  \ll_j
  (1+\mathcal L_p)^{J+j}.
\end{equation}
Moreover, for every \(\tau,R>0\) and all integers \(k\geq 0\),
uniformly for
\(
  1\leq \Re w\leq 1+\tau,
  \)
  \(|\Im w|\leq R,
\)
one has
\begin{align}
  % \bigl|d_p^{(j)}(w)\bigr|
  % &\ll_{j,\tau,R}
  % \frac{(1+\mathcal L_p)^{j}}{p^2},
  % \label{eq:dp-derivative-bound}\\
  \bigl|T_p^{(k)}(w)\bigr|
  &\ll_{k,\tau,R}
  \frac{(1+\mathcal L_p)^{J+k}}{p^2}.
  \label{eq:Tp-derivative-bound}
\end{align}
\end{lemma}

\begin{proof}
We first prove \eqref{eq:Vp-derivative-bound}. The prime \(p=2\)
causes no difficulty, so suppose that \(p\geq 3\). Fix \(w\) with
\(\Re w\geq 1\), and consider the circle
\[
  |z-w|=\frac{1}{\mathcal L_p}.
\]
For every \(z\) on this circle and every prime \(q<p\), we have
\[
  |q^{-z}|
  =q^{-\Re z}
  \leq \frac{q^{1/\mathcal L_p}}{q}
  \leq \frac{\mathrm e}{q}.
\]
It follows that
\[
  |V_p(z)|
  \leq
  \prod_{q<p}\bigl(1+|q^{-z}|\bigr)
  \leq
  \exp\left(\mathrm e\sum_{\substack{q<p\\q\text{ prime}}}\frac1q\right).
\]
By Mertens' estimate for the sum of the reciprocals of the primes,
there is a fixed integer \(J\geq 0\) such that
\[
  |V_p(z)|\ll (1+\mathcal L_p)^J.
\]
Cauchy's estimate on the circle of radius \(1/\mathcal L_p\) therefore
gives
\[
  |V_p^{(j)}(w)|
  \leq
  j!\,\mathcal L_p^j
  \max_{|z-w|=1/\mathcal L_p}|V_p(z)|
  \ll_j (1+\mathcal L_p)^{J+j}.
\]

We next estimate the derivatives of \(d_p\). Differentiating $j$ times
gives
\[
  d_p^{(j)}(w)
  =
  (-\log p)^j p^{-w}
  -
  \bigl(-\log(p+1)\bigr)^j(p+1)^{-w}.
\]
Fix \(j\) and \(w\). For $x\geq 2$, define
\[
  f_{j,w}(x):=(-\log x)^j x^{-w}.
\]
Then
\[
  d_p^{(j)}(w)=f_{j,w}(p)-f_{j,w}(p+1).
\]
For complex \(w\), we have the identity
\[
  d_p^{(j)}(w)
  =
  -\int_p^{p+1}f_{j,w}'(x)\,dx.
\]
A direct calculation gives
\[
  f_{j,w}'(x)
  =
  (-1)^j x^{-w-1}
  \left(
    j(\log x)^{j-1}-w(\log x)^j
  \right),
\]
where the first term is omitted when \(j=0\). If
\(1\leq\Re w\leq1+\tau\) and \(|\Im w|\leq R\), then
\(|w|\ll_{\tau,R}1\). Moreover, for \(p\leq x\leq p+1\),
\[
  x^{-\Re w-1}\leq p^{-2},
  \qquad
  1+\log x\ll 1+\mathcal L_p.
\]
Consequently,
\[
  |f_{j,w}'(x)|
  \ll_{j,\tau,R}
  \frac{(1+\mathcal L_p)^{j}}{p^2}.
\]
Integration over \([p,p+1]\) proves
% \eqref{eq:dp-derivative-bound}.
\begin{align}
    \bigl|d_p^{(j)}(w)\bigr|
  &\ll_{j,\tau,R}
  \frac{(1+\mathcal L_p)^{j}}{p^2}.
  \label{eq:dp-derivative-bound}
\end{align}

Finally, Leibniz's rule gives
\[
  T_p^{(k)}(w)
  =
  \sum_{r=0}^k
  \binom{k}{r}
  d_p^{(r)}(w)V_p^{(k-r)}(w).
\]
Substituting the bounds in \eqref{eq:Vp-derivative-bound} and
\eqref{eq:dp-derivative-bound} proves \eqref{eq:Tp-derivative-bound}.
\end{proof}
% Since
% \[
%   \sum_p\frac{(1+\log p)^A}{p^2}<\infty
% \]
% for every fixed \(A\geq0\), Lemma~\ref{lem:derivative-estimates}
% implies that, for every \(k\geq0\), the series
% \[
%   \sum_p T_p^{(k)}(w)
% \]
% converges uniformly on compact subsets of the closed half-plane
% \(\Re w\geq1\).  In particular, for \(\Re w>1\) we may differentiate
% \eqref{eq:prime-decomposition} term by term to obtain
% \[
%   D^{(k)}(w)=\sum_pT_p^{(k)}(w).
% \]
% We henceforth use the same notation \(D^{(k)}\) for its continuous
% extension to \(\Re w\geq1\), namely
% \begin{equation}\label{eq:D-derivative-prime-series}
%   D^{(k)}(w):=\sum_pT_p^{(k)}(w),
%   \qquad \Re w\geq1.
% \end{equation}
For every compact set
\[
  K\subseteq \{w\in\C:\Re w\geq 1\},
\]
choose \(\tau,R>0\) such that
\[
  K\subseteq
  \{w:1\leq \Re w\leq 1+\tau,\ |\Im w|\leq R\}.
\]
By Lemma~\ref{lem:derivative-estimates},
\[
  \sup_{w\in K}|T_p^{(k)}(w)|
  \ll_{k,\tau, R}\frac{(1+\log p)^{J+k}}{p^2}.
\]
Since
\[
  \sum_p\frac{(1+\log p)^{J+k}}{p^2}<\infty,
\]
the series
\[
  \sum_p T_p^{(k)}(w)
\]
converges absolutely and uniformly on \(K\). Hence, by the
Weierstrass theorem, the function \(\sum_pT_p(w)\) is holomorphic in
\(\Re w>1\), and therefore
\[
  D^{(k)}(w)=\sum_pT_p^{(k)}(w).
\]

We henceforth use the same notation \(D^{(k)}\) for its continuous
extension to \(\Re w\geq1\), namely
for \(k\geq0\), put
\[
  D^{(k)}(w):=\sum_pT_p^{(k)}(w),
  \qquad \Re w\geq1.
\]
For \(0<a<b\), the fundamental
theorem of calculus gives
\[
  D^{(k)}(1+b)-D^{(k)}(1+a)
  =\int_a^bD^{(k+1)}(1+u)\,du.
\]
Letting \(a\downarrow0\), we obtain
\[
  D^{(k)}(1+b)-D^{(k)}(1)
  =\int_0^bD^{(k+1)}(1+u)\,du.
\]
Consequently, \(x\mapsto D(1+x)\) has a \(C^\infty\) extension to
\([0,\infty)\), and
\[
  \left.\frac{d^k}{dx^k}D(1+x)\right|_{x=0+}
  =D^{(k)}(1)=\sum_pT_p^{(k)}(1).
\]

For every integer \(k\geq0\), put
\begin{equation}\label{eq:def-Dk}
  D_k:=D^{(k)}(1)
      =\sum_pT_p^{(k)}(1).
\end{equation}
In particular,
\[
  D(1)=D_0=\sum_pT_p(1)=\Xi.
\]

% Since
% \[
%   \sum_p\frac{(1+\log p)^A}{p^2}<\infty
% \]
% for every fixed \(A\geq 0\), Lemma~\ref{lem:derivative-estimates}
% implies that, for every \(k\geq 0\), the series
% \[
%   \sum_p T_p^{(k)}(w)
% \]
% converges uniformly on compact subsets of the closed half-plane
% \(\Re w\geq 1\). In particular, it converges locally uniformly in
% \(\Re w>1\). We may therefore differentiate
% \eqref{eq:prime-decomposition} term by term when $\Re w>1$ to obtain
% \begin{equation}\label{eq:D-derivative-prime-series}
%   D^{(k)}(w)=\sum_p T_p^{(k)}(w).
% \end{equation}

% For every integer \(k\geq 0\), define
% \begin{equation}\label{eq:def-Dk}
%   D_k:=\sum_p T_p^{(k)}(1).
% \end{equation}
% In particular, 
% \begin{align*}
%     D_0 = \sum_{p}T_p(1) = \Xi.
% \end{align*}
% The series in~\eqref{eq:def-Dk} is absolutely convergent by
% \eqref{eq:Tp-derivative-bound}. Taking \(w=\sigma\) real in
% \eqref{eq:D-derivative-prime-series} and applying dominated
% convergence on any fixed interval \(1\leq\sigma\leq1+\tau\), we obtain
% \begin{equation}\label{eq:def-Dk}
%   D^{(k)}(\sigma)=\sum_p T_p^{(k)}(\sigma),
%   \qquad
%   \lim_{\sigma\downarrow1}D^{(k)}(\sigma)=D_k.
% \end{equation}

Note that 
\[
  T_p(1)
  =\frac{1}{p(p+1)}
   \prod_{q<p}\left(1-\frac1q\right)>0.
\]
Therefore, for each prime \(p\), we may
define
\begin{equation}\label{eq:normalized-Gp}
  G_p(z):=
  \frac{T_p(1+z/\mathcal L_p)}{T_p(1)}.
\end{equation}

Repeated application of the chain rule gives
\[
  G_p^{(k)}(z)
  =
  \frac{1}{T_p(1)\mathcal L_p^k}
  T_p^{(k)}\left(1+\frac{z}{\mathcal L_p}\right).
\]
Evaluating at \(z=0\), we find
\[
  T_p^{(k)}(1)
  =
  T_p(1)\mathcal L_p^kG_p^{(k)}(0).
\]
Substituting this identity into \eqref{eq:def-Dk} yields
\begin{equation*}\label{eq:Dk-prime-sum}
  D_k
  =
  \sum_p
  T_p(1)\mathcal L_p^kG_p^{(k)}(0).
\end{equation*}
We conclude this section by isolating the two estimates for \(D_k\) that
will drive the proof of Theorem~\ref{thm:main}.  Their proofs occupy the
subsequent sections.

\begin{proposition}
\label{prop:Dk-upper}
Let \(\alpha>0\), and suppose that \(C>0\) is fixed.

If
\[
  s_1(n)-\Xi n\leq C R_\alpha(n)
\]
for all sufficiently large \(n\), then, for every
\(0<\beta<\alpha\) and for all sufficiently large \(k\),
\begin{equation}\label{eq:Dk-upper-positive}
  (-1)^{k-1}D_k
  \leq
  \e^{o(k)}\left(\frac{2}{\beta^2}\right)^k
  \left(
    \frac{k^2}{e^2\log k}
  \right)^k.
\end{equation}
If
\[
  s_1(n)-\Xi n\geq-C R_\alpha(n)
\]
for all sufficiently large \(n\), then, under the same conditions,
\begin{equation}\label{eq:Dk-upper-negative}
  (-1)^{k-1}D_k
  \geq
  -\e^{o(k)}\left(
    \frac{2}{\beta^2}
  \right)^k
  \left(
    \frac{k^2}{e^2\log k}
  \right)^k.
\end{equation}
\end{proposition}

\begin{proposition}
\label{prop:Dk-lower}
There are sequences of positive integers
\(k_j^+,k_j^-\to\infty\) such that, for all sufficiently large \(j\),
\begin{align}
  (-1)^{k_j^+-1}D_{k_j^+}
  &\geq
  e^{o(k_j^+)}\left(
    \frac{(k_j^+)^2}{e^2\log k_j^+}
  \right)^{k_j^+},
  \label{eq:Dk-lower-positive}\\
  (-1)^{k_j^--1}D_{k_j^-}
  &\leq
  -e^{o(k_j^-)}\left(
    \frac{(k_j^-)^2}{e^2\log k_j^-}
  \right)^{k_j^-}.
  \label{eq:Dk-lower-negative}
\end{align}
\end{proposition}        

\begin{proof}[Proof of Theorem~\ref{thm:main}, assuming
Propositions~\ref{prop:Dk-upper} and~\ref{prop:Dk-lower}]
Fix \(\alpha>\sqrt2\), and choose \(\beta\) with
\[
  \sqrt2<\beta<\alpha.
\]
% Choose \(\delta,\varepsilon>0\) sufficiently small that
% \begin{equation}\label{eq:separation-ratio}
%   \varrho:=\e^\varepsilon(1+\delta)\frac{2}{\beta^2}<1.
% \end{equation}

Suppose first that
\[
  \limsup_{n\to\infty}\frac{s_1(n)-\Xi n}{R_\alpha(n)}<\infty.
\]
Then \(s_1(n)-\Xi n\leq C R_\alpha(n)\) for some \(C>0\) and all sufficiently
large \(n\).  Apply \eqref{eq:Dk-upper-positive} with
\(k=k_j^+\), and compare it with \eqref{eq:Dk-lower-positive}.  The
ratio of the upper bound to the lower bound is at most
\[
  \left(\frac{e^{o(1)}2}{\beta^2}\right)^k,
\]
which tends to zero as $\beta>\sqrt{2}$.  This is impossible, and hence
\[
  \limsup_{n\to\infty}\frac{s_1(n)-\Xi n}{R_\alpha(n)}=\infty.
\]

Similarly, if
\[
  \liminf_{n\to\infty}\frac{s_1(n)-\Xi n}{R_\alpha(n)}> -\infty,
\]
then \(s_1(n)-\Xi n\geq-C R_\alpha(n)\) eventually.  Comparing
\eqref{eq:Dk-upper-negative} with \eqref{eq:Dk-lower-negative} along
\(k=k_j^-\) gives the same contradiction.  Therefore
\[
  \liminf_{n\to\infty}\frac{s_1(n)-\Xi n}{R_\alpha(n)}=-\infty.
\]
This proves the theorem.
\end{proof}
\section{Upper Bound on \texorpdfstring{\(D_k\)}{Dk}}\label{sec:upper-bound-Dk}

Put, for $t\geq 0$,
\[
  E(n):=s_1(n)-\Xi n,
  \qquad
  h(t):=\e^{-t}E(\lfloor\e^t\rfloor).
\]
Let \(\alpha>0\).  In this section we assume one of the two one-sided
bounds
\begin{equation}\label{eq:one-sided-error-bounds}
  E(n)\leq C R_\alpha(n)
  \qquad\text{or}\qquad
  E(n)\geq-C R_\alpha(n)
\end{equation}
for all sufficiently large \(n\), where \(C>0\) is fixed.  If
\(0<\beta<\alpha\), then comparison of
\(\log\lfloor\e^t\rfloor\) with \(t\) shows that, for all sufficiently
large \(t\), respectively,
\begin{equation}\label{eq:one-sided-h-bounds}
  h(t)\leq C_\beta\e^{-\beta\sqrt{t\log t}}
  \qquad\text{or}\qquad
  h(t)\geq-C_\beta\e^{-\beta\sqrt{t\log t}}.
\end{equation}
The strict inequality \(\beta<\alpha\) permits the replacement of
\(\log\lfloor\e^t\rfloor\) by \(t\) in the exponent.

For \(\Re z>0\), define
\begin{equation}\label{eq:Mellin-error-def}
  \mathcal M_E(z)
  :=
  \int_1^\infty E(\lfloor x\rfloor)x^{-z-2}\dd x
  =\int_0^\infty h(t)\e^{-zt}\dd t.
\end{equation}
From~\eqref{eq:dirichlet-convolution}, we see that the partial sums of the sequence \(b(n)-\Xi \) satisfy
\[
  \sum_{n\leq x}(b(n)-\Xi )
  =s_1(\lfloor x\rfloor)-\Xi\lfloor x\rfloor
  =E(\lfloor x\rfloor).
\]
For $\Re z>0$, partial summation and \eqref{eq:dirichlet-product} therefore give
\begin{equation*}\label{eq:centered-Mellin}
  \zeta(1+z)\bigl(D(1+z)-\Xi \bigr)
  =(1+z)\mathcal M_E(z).
\end{equation*}
Let \(Z(w)\) denote the holomorphic extension of \((w-1)\zeta(w)\) to a
neighborhood of \(w=1\). Choose \(0<r_0<1\) so that \(Z(1+z)\neq0\) for \(|z|<r_0\),
and put
\[
Q(z):=\frac{1+z}{Z(1+z)},\qquad \Phi(z):=zQ(z).
\]
Then \(Q\) is holomorphic on this disk and \(Q(0)=1\).
For \(|z|<r_0\) with \(\Re z>0\),
\begin{equation}\label{eq:factorized-Mellin}
D(1+z)-\Xi=\Phi(z)\mathcal M_E(z).
\end{equation}
%and put
% \[
%   Q(z):=\frac{1+z}{Z(1+z)},
%   \qquad
%   \Phi(z):=zQ(z).
% \]
% Then \(Q\) is holomorphic near \(0\), \(Q(0)=1\), and for $\Re z>0$,
% \begin{equation}\label{eq:factorized-Mellin}
%   D(1+z)-\Xi =\Phi(z)\mathcal M_E(z).
% \end{equation}

The next lemma explains why a one-sided hypothesis is sufficient.  It
will also justify all differentiations at \(z=0\).
\begin{lemma}
\label{lem:one-sided-moments}
Assume either of the bounds in \eqref{eq:one-sided-error-bounds}.  Then,
for every integer \(m\geq0\),
\begin{equation*}\label{eq:absolute-log-moments}
  \int_0^\infty t^m|h(t)|\dd t<\infty,
\end{equation*}
and \(\mathcal M_E\) has a \(C^\infty\) extension to \([0,\infty)\)
satisfying
\begin{equation*}\label{eq:Mellin-moment-identity}
  \mathcal M_E^{(m)}(0)
  =(-1)^m\int_0^\infty t^m h(t)\dd t.
\end{equation*}
\end{lemma}

% \begin{proof}
% The trivial bound \(|E(n)|\ll n\) shows that \(h\) is bounded.  Hence,
% for \(z>0\), differentiation under the integral in
% \eqref{eq:Mellin-error-def} gives
% \[
%   \mathcal M_E^{(m)}(z)
%   =(-1)^m\int_0^\infty t^m h(t)\e^{-zt}\dd t.
% \]
% By \eqref{eq:density-constant} and the definition of \(T_p\),
% \(D_0=\sum_pT_p(1)=\Xi\).  The function \(D(1+z)-\Xi\) therefore vanishes
% at \(z=0\), and \eqref{eq:def-Dk} shows that it has
% right derivatives of every order there.  Since \(\Phi(z)=zQ(z)\) and
% \(Q(0)=1\), the quotient on the right of
% \[
%   \mathcal M_E(z)=\frac{D(1+z)-\Xi}{\Phi(z)}
% \]
% extends to a \(C^\infty\) function at \(z=0\).  In particular,
% \(\mathcal M_E^{(m)}(z)\) has a finite limit as \(z\downarrow0\).

% Suppose first that the upper bound in \eqref{eq:one-sided-h-bounds}
% holds.  Then every moment of the positive part \(h_+ := \max\{h, 0\}\) is finite.  If
% \(h_-:=\max\{-h,0\}\), the identity above gives
% \[
%   \int_0^\infty t^m h_-(t)\e^{-zt}\dd t
%   =
%   \int_0^\infty t^m h_+(t)\e^{-zt}\dd t
%   -(-1)^m\mathcal M_E^{(m)}(z).
% \]
% The right-hand side has a finite limit as \(z\downarrow0\).  Monotone
% convergence therefore shows that the \(m\)-th moment of \(h_-\) is
% finite.  This proves \eqref{eq:absolute-log-moments} and then
% \eqref{eq:Mellin-moment-identity} by dominated convergence.  Under the
% lower bound in \eqref{eq:one-sided-h-bounds}, the same argument is
% applied with \(h_+\) and \(h_-\) interchanged.
% \end{proof}
\begin{proof}
The trivial estimate \(|E(n)|\ll n\) gives
\[
  |h(t)|=e^{-t}|E(\lfloor e^t\rfloor)|\ll1.
\]
Consequently, for every \(z>0\) and every integer \(m\geq0\), we may
differentiate under the integral in~\eqref{eq:Mellin-error-def}, obtaining
\[
  \mathcal M_E^{(m)}(z)
  =(-1)^m\int_0^\infty t^m h(t)e^{-zt}\,dt.
\]
Indeed, locally uniformly for \(z>0\), the integrand is dominated by
a constant multiple of \(t^m e^{-ct}\) for some \(c>0\).

By the discussion following Lemma~\ref{lem:derivative-estimates}, the function
\[
  F(x):=D(1+x)-\Xi,\qquad x\geq0,
\]
belongs to \(C^\infty([0,\infty))\), and \(F(0)=0\). Since
\(\Phi(x)=xQ(x)\) and \(Q(0)=1\), after restricting to a sufficiently
small neighborhood of \(0\) we have \(Q(x)\neq0\). Moreover,
\begin{equation}\label{eq:F-over-x}
  \frac{F(x)}x=\int_0^1F'(ux)\,du
  \qquad (x>0).
\end{equation}
The right-hand side defines a \(C^\infty\) function also at \(x=0\);
indeed, for every \(r\geq0\),
\[
  \frac{d^r}{dx^r}\int_0^1F'(ux)\,du
  =\int_0^1u^rF^{(r+1)}(ux)\,du.
\]
It follows that
\[
  \widetilde{\mathcal M}_E(x)
  :=\frac1{Q(x)}\int_0^1F'(ux)\,du
\]
is \(C^\infty\) near \(0\). By~\eqref{eq:factorized-Mellin} and~\eqref{eq:F-over-x}, it agrees with \(\mathcal M_E(x)\)
for \(x>0\), and hence gives a \(C^\infty\) extension of \(\mathcal M_E\) to
\(x=0\). In particular, \(\mathcal M_E^{(m)}(z)\) has a finite limit as
\(z\downarrow0\).

Suppose first that the upper bound in~\eqref{eq:one-sided-h-bounds} holds. Put
\[
  h_+(t):=\max\{h(t),0\},
  \qquad
  h_-(t):=\max\{-h(t),0\}.
\]
Then, for all sufficiently large \(t\),
\[
  0\leq h_+(t)
  \leq C_\beta e^{-\beta\sqrt{t\log t}}.
\]
Since the function on the right has finite polynomial moments, we have
\[
  \int_0^\infty t^m h_+(t)\,dt<\infty.
\]
For \(z>0\),
\[
  \int_0^\infty t^m h_-(t)e^{-zt}\,dt
  =
  \int_0^\infty t^m h_+(t)e^{-zt}\,dt
  -(-1)^m\mathcal M_E^{(m)}(z).
\]
The right-hand side has a finite limit as \(z\downarrow0\).
Monotone convergence on the left therefore gives
\[
  \int_0^\infty t^m h_-(t)\,dt<\infty.
\]
Hence
\[
  \int_0^\infty t^m|h(t)|\,dt<\infty.
\]
Dominated convergence now yields
\[
  \mathcal M_E^{(m)}(0)
  =\lim_{z\downarrow0}\mathcal M_E^{(m)}(z)
  =(-1)^m\int_0^\infty t^mh(t)\,dt.
\]

Under the lower bound in~\eqref{eq:one-sided-h-bounds}, the same argument applies with
\(h_+\) and \(h_-\) interchanged.
\end{proof}

For \(k\geq1\), define
\begin{equation*}\label{eq:polynomial-weight}
  \mathcal{K}_k(t)
  :=
  \left.\frac{\dd^k}{\dd z^k}
    \bigl(\Phi(z)\e^{-zt}\bigr)\right|_{z=0}.
\end{equation*}
The following identity expresses \(D_k\) in terms of the Mellin transform
\(\mathcal M_E\) of \(h(t)\) and the polynomial weight
\(\mathcal{K}_k(t)\).  This representation will allow us to exploit the
bounds on \(h(t)\) in \eqref{eq:one-sided-h-bounds}.
\begin{lemma}
\label{lem:integral-representation}
Under either hypothesis in \eqref{eq:one-sided-error-bounds}, for every
\(k\geq1\),
\begin{equation}\label{eq:integral-representation}
  D_k=\int_0^\infty h(t)\mathcal{K}_k(t)\dd t.
\end{equation}
Moreover, \(\mathcal{K}_k\) is a polynomial of degree at most \(k-1\).
\end{lemma}
\begin{proof}
By Lemma~\ref{lem:one-sided-moments}, \(\mathcal M_E\), regarded as a
function of the real variable \(z\geq0\), has a \(C^\infty\) extension to
\(z=0\). By the discussion following Lemma~\ref{lem:derivative-estimates},
the same is true of \(D(1+z)\), and
\[
  \left.\frac{d^k}{dz^k}D(1+z)\right|_{z=0+}=D_k.
\]
Hence both sides of~\eqref{eq:factorized-Mellin} have \(C^\infty\) extensions to \(z=0\).
Since they agree for \(z>0\), their \(k\)-th right derivatives at
\(z=0\) agree. Thus, for \(k\geq1\),
\[
  D_k
  =
  \sum_{j=0}^k
  \binom{k}{j}
  \Phi^{(j)}(0)\mathcal M_E^{(k-j)}(0).
\]
By Lemma~\ref{lem:one-sided-moments},
\[
  \mathcal M_E^{(k-j)}(0)
  =
  (-1)^{k-j}
  \int_0^\infty t^{k-j}h(t)\,dt.
\]
All these integrals are absolutely convergent, so the finite sum may
be moved inside the integral. Therefore
\[
  D_k
  =
  \int_0^\infty h(t)
  \left\{
    \sum_{j=0}^k
    \binom{k}{j}
    \Phi^{(j)}(0)(-t)^{k-j}
  \right\}\,dt
  =
  \int_0^\infty h(t)\mathcal{K}_k(t)\,dt.
\]
Finally,
\[
  \mathcal{K}_k(t)
  =
  \sum_{j=0}^k
  \binom{k}{j}
  \Phi^{(j)}(0)(-t)^{k-j}.
\]
Since \(\Phi(0)=0\), the term with \(j=0\), which would have degree
\(k\), vanishes. Hence \(\mathcal{K}_k\) has degree at most \(k-1\).
\end{proof}

% \begin{proof}
% By Lemma~\ref{lem:one-sided-moments}, every polynomial moment of
% \(|h|\) is finite.  We may therefore differentiate
% \eqref{eq:factorized-Mellin} \(k\) times, use
% \eqref{eq:Mellin-error-def}, and let \(z\downarrow0\).  Dominated
% convergence gives the right-hand side of
% \eqref{eq:integral-representation}, while
% \eqref{eq:def-Dk} gives \(D_k\) on the left.  Finally,
% Leibniz's rule gives
% \[
%   \mathcal{K}_k(t)
%   =\sum_{j=0}^k\binom{k}{j}\Phi^{(j)}(0)(-t)^{k-j}.
% \]
% Since \(\Phi(0)=0\), the term of degree \(k\) vanishes.
% \end{proof}

The zero of \(\Phi\) at the origin also fixes the eventual sign of the
kernel.

\begin{lemma}
\label{lem:kernel-sign}
There are constants \(A_0,C_0>0\) such that, for every \(k\geq1\) and
\(t\geq A_0k\),
\begin{equation*}\label{eq:kernel-sign-bound}
  0\leq(-1)^{k-1}\mathcal{K}_k(t)
  \leq C_0k t^{k-1}.
\end{equation*}
\end{lemma}

\begin{proof}
Since \(\Phi'(0)=Q(0)=1\), Leibniz's rule gives
\[
  (-1)^{k-1}\mathcal{K}_k(t)
  =kt^{k-1}
   +\sum_{j=2}^k\binom{k}{j}\Phi^{(j)}(0)(-1)^{j-1}t^{k-j}.
\]
Choose \(R>0\) such that \(\Phi\) is holomorphic on a neighborhood of
\(|z|\leq R\), and let \(M=\max_{|z|=R}|\Phi(z)|\).  Cauchy's estimate
gives \(|\Phi^{(j)}(0)|\leq Mj!R^{-j}\), and hence the absolute value of
the sum over \(j\geq2\), divided by \(kt^{k-1}\), is at most
\[
  \frac{M}{R}\sum_{m\geq1}\left(\frac{k}{Rt}\right)^m.
\]
Taking \(A_0\) sufficiently large makes this quantity at most
\(1/2\).  The assertion follows, after enlarging the upper constant.
\end{proof}

For \(\beta>0\), put
\[
  J_m(\beta):=\int_1^\infty
  t^m\e^{-\beta\sqrt{t\log t}}\dd t.
\]

\begin{lemma}
\label{lem:J-upper}
Let \(\beta>0\).  For all sufficiently large
\(k\),
\begin{equation}\label{eq:J-upper}
  J_{k-1}(\beta)
  \leq 
  \e^{o_\beta(k)}
  \left(
    \frac{2}{\beta^2\e^2}\frac{k^2}{\log k}
  \right)^k.
\end{equation}
\end{lemma}

\begin{proof}
For \(t,T\geq1\), the inequality
\(\log(T/t)\leq T/t-1\) gives
\[
  (1+\log T)t\leq t\log t+T,
\]
and consequently
\[
  \sqrt{t\log t}
  \geq\sqrt{1+\log T}\sqrt t-\sqrt T.
\]
It follows that
\begin{align}
  J_{k-1}(\beta)
  &\leq
  \e^{\beta\sqrt T}
  \int_0^\infty t^{k-1}
  \e^{-\beta\sqrt{1+\log T}\sqrt t}\dd t \notag\\
  &=
  \e^{\beta\sqrt T}
  \frac{2(2k-1)!}{\{\beta^2(1+\log T)\}^k}.
  \label{eq:gamma-majorant}
\end{align}
Taking \(T=k^2/\log k\) and using Stirling's upper bound

% \[
%   \beta\sqrt T=o(k),
%   \qquad
%   1+\log T=2\log k-\log\log k+1,
% \]
% and
% \[
%   \frac{2\log k}{1+\log T}\longrightarrow1.
% \]
% Stirling's upper bound now gives, for all sufficiently large \(k\),
\begin{align*}
  2(2k-1)!&\leq
  \e^{o(k)}\left(\frac{2k}{\e}\right)^{2k}
  % \e^{\beta\sqrt T}&\leq\e^{\varepsilon k/3},\\
  % \left(\frac{2\log k}{1+\log T}\right)^k
  % &\leq\e^{\varepsilon k/3}.
\end{align*}
in \eqref{eq:gamma-majorant} proves
\eqref{eq:J-upper}.
\end{proof}
\begin{remark}
A more precise oscillation result should be obtainable by carrying the analysis of the integral defining $J_m(\beta)$ to further orders, for instance by a more careful saddle-point analysis. This would introduce explicit second-order corrections to the oscillation scale and allow the constant in the exponential to approach the critical value $\sqrt{2}$ more quantitatively. We do not pursue such refinements here, since they considerably complicate the analysis while not resolving the main remaining issue: the endpoint corresponding to the constant $\sqrt{2}$ itself. 
\end{remark}

We now have all the ingredients to prove Proposition~\ref{prop:Dk-upper}.
\begin{proof}[Proof of Proposition~\ref{prop:Dk-upper}]
Put
\[
  \mathcal L_k(t):=(-1)^{k-1}\mathcal{K}_k(t),
  \qquad T:=A_0k,
\]
where \(A_0\) is supplied by Lemma~\ref{lem:kernel-sign}.  On
\([0,T]\), Cauchy's estimate on a fixed circle \(|z|=R\) gives
\[
  |\mathcal{K}_k(t)|
  \leq \frac{k!}{R^k}
  \max_{|z|=R}|\Phi(z)|\e^{Rt}.
\]
Since \(h\) is bounded, the contribution of the interval $[0, T]$ to
\eqref{eq:integral-representation} satisfies
\begin{equation*}\label{eq:initial-kernel-range}
  \left|\int_0^T h(t)\mathcal L_k(t)\dd t\right|\leq \exp\{k\log k+O(k)\}.
\end{equation*}
This is smaller than the right-hand side in either conclusion of
Proposition~\ref{prop:Dk-upper} for all sufficiently large \(k\).

Suppose first that \(E(n)\leq C R_\alpha(n)\) eventually.  By
\eqref{eq:one-sided-h-bounds} and
Lemma~\ref{lem:kernel-sign},
\begin{align*}
  (-1)^{k-1}D_k
  &\leq
  \left|\int_0^T h(t)\mathcal L_k(t)\dd t\right|
  +C_{\beta}C_0k
   \int_T^\infty t^{k-1}
   \e^{-\beta\sqrt{t\log t}}\dd t\\
  &\leq
  \exp\{k\log k+O(k)\}
  +C_{\beta}C_0kJ_{k-1}(\beta).
\end{align*}
Lemma~\ref{lem:J-upper}, with a harmless enlargement of the constant
and of the polynomial factor, proves the first assertion of
Proposition~\ref{prop:Dk-upper}.  If
\(E(n)\geq-C R_\alpha(n)\) eventually, the same argument applied to
\(-h\) gives the second assertion.
\end{proof}
\section{Saddle-Point Analysis of the Limiting Function}\label{sec:saddle-analysis}

Define the complementary exponential integral and the limiting entire
function by
\begin{equation*}\label{eq:limiting-entire-function}
  \Ein(z)
  =\sum_{m\geq1}\frac{(-1)^{m+1}z^m}{m\,m!}
  =\int_0^1\frac{1-\e^{-zt}}{t}\dd t,
  \qquad
  f(z):=\exp\{-z+\Ein(z)\}.
\end{equation*}
At \(t=0\), the integrand is interpreted by continuity. Put
\[
  H(z):=-z+\Ein(z),
\]
so that \(f(z)=\e^{H(z)}\).

We will see in the next section that $G_p(z)$ can be approximated by $f(z)$ for large $p$. In this section we analyze the function $f(z)$ in detail, and in particular we will obtain saddle-point estimates for $f(z)$ and its derivatives.

We first record the exponential-integral formulas used below. Let
\[
  \Omega:=\C\setminus (-\infty,0],
\]
and let
\[
  \Log z:=\log|z|+i\arg z,
  \qquad -\pi<\arg z<\pi,
\]
denote the principal logarithm. For \(z\in\Omega\), define
\begin{equation}\label{eq:E1-definition}
  E_1(z):=\e^{-z}\int_0^\infty\frac{\e^{-t}}{t+z}\dd t.
\end{equation}
In what follows, let $\gamma$ denote the Euler–Mascheroni constant.
\begin{lemma}
\label{lem:E1-formulas}
For \(z\in\Omega\),
\begin{equation}\label{eq:E1-Ein-connection}
  E_1(z)=\Ein(z)-\Log z-\gamma.
\end{equation}
If \(\pi/2\leq|\arg z|<\pi\), then
\begin{equation*}\label{eq:E1-three-term-expansion}
  E_1(z)
  =
  \frac{\e^{-z}}{z}
  \left(
    1-\frac1z+\frac{2}{z^2}
  \right)
  +R_3(z),
\end{equation*}
where
\begin{equation}\label{eq:E1-three-term-remainder-bound}
  |R_3(z)|
  \leq
  \frac{6\e^{-\Re z}}
  {|z|^4|\sin(\arg z)|}.
\end{equation}
\end{lemma}

\begin{proof}
The integral in \eqref{eq:E1-definition} converges locally uniformly on
\(\Omega\). Differentiating under the integral gives
\[
  E_1'(z)=-\frac{\e^{-z}}{z},
  \qquad
  \Ein'(z)=\frac{1-\e^{-z}}{z}.
\]
Hence \(E_1(z)-\Ein(z)+\Log z\) is constant. For \(x>0\),
\[
  E_1(x)=\int_x^\infty\frac{\e^{-t}}{t}\dd t,
\]
and therefore
\begin{align*}
  \lim_{x\downarrow0}\bigl(E_1(x)+\log x\bigr)
  &=
  \int_0^1\frac{\e^{-t}-1}{t}\dd t
  +\int_1^\infty\frac{\e^{-t}}{t}\dd t\\
  &=-\gamma.
\end{align*}
The last equality is the standard limiting formula for the
Euler--Mascheroni constant; see, for instance,
\cite[\S5.1]{AbramowitzStegun1964}.
Since \(\Ein(x)\to0\) as \(x\downarrow0\), this proves
\eqref{eq:E1-Ein-connection}.

For the second assertion, use the exact identity
\[
  \frac1{t+z}
  =
  \frac1z-\frac{t}{z^2}+\frac{t^2}{z^3}
  -\frac{t^3}{z^3(t+z)}.
\]
Substitution into \eqref{eq:E1-definition} gives
\begin{align*}
  E_1(z)
  &=
  \e^{-z}
  \left\{
    \frac1z\int_0^\infty\e^{-t}\dd t
    -\frac1{z^2}\int_0^\infty t\e^{-t}\dd t
    +\frac1{z^3}\int_0^\infty t^2\e^{-t}\dd t
  \right\}
  +R_3(z)\\
  &=
  \frac{\e^{-z}}{z}
  \left(
    1-\frac1z+\frac{2}{z^2}
  \right)
  +R_3(z),
\end{align*}
where
\[
  R_3(z)
  =
  -\frac{\e^{-z}}{z^3}
  \int_0^\infty
  \frac{t^3\e^{-t}}{t+z}\dd t.
\]

If \(z=|z|\e^{i\theta}\) and
\(\pi/2\leq|\theta|<\pi\), then
\[
  |t+z|
  \geq|\Im z|
  =
  |z\sin\theta|
  \qquad (t\geq0).
\]
Consequently,
\begin{align*}
  |R_3(z)|
  &\leq
  \frac{\e^{-\Re z}}{|z|^3}
  \int_0^\infty
  \frac{t^3\e^{-t}}{|t+z|}\dd t\leq
  \frac{\e^{-\Re z}}
  {|z|^4|\sin\theta|}
  \int_0^\infty t^3\e^{-t}\dd t=
  \frac{6\e^{-\Re z}}
  {|z|^4|\sin\theta|},
\end{align*}
which proves \eqref{eq:E1-three-term-remainder-bound}. 
% The same
% asymptotic expansion is recorded in Abramowitz and
% Stegun~\cite[\S5.1]{AbramowitzStegun1964} and
% Olver~\cite[Chapter~1, \S3, and pp.~109--112]{Olver1997}.
\end{proof}
Our analysis would require us to understand the derivatives of $f(z)$. We will use Cauchy's integral formula along with a saddle-point analysis to get a satisfactory approximation for $f^{(k)}(0)$ as $k \to \infty$. It will turn out that the saddle equation for the \(k\)-th Taylor coefficient of $f(z)$ is
\[
  zH'(z)=1-z-\e^{-z}=k.
\]
It is also convenient to write \(z=-\xi\).

\begin{lemma}
\label{lem:saddle-points}
For all sufficiently large real \(x\), the equation
\begin{equation}\label{eq:saddle-equation-xi}
  \e^\xi=1+\xi-x
\end{equation}
has a unique solution \(\xi(x)\) satisfying
\[
  \bigl|\xi(x)-(\log x+i\pi)\bigr|
  \leq\frac{4\log x}{x}.
\]
Moreover,
\begin{align}
  \xi(x)
  &=\log x+i\pi
    -\frac{1+\log x+i\pi}{x}
    +O\!\left(\frac{(\log x)^2}{x^2}\right),
    \label{eq:xi-expansion}\\
  \xi'(x)
  &=\frac1{x-\xi(x)}.
    \label{eq:xi-derivative}
\end{align}
\end{lemma}

\begin{proof}
Write
\[
  \xi=\log x+i\pi+u.
\]
Equation \eqref{eq:saddle-equation-xi} is equivalent to
\[
  u=
  \Log\left(1-\frac{1+\log x+i\pi+u}{x}\right),
\]
where the logarithm is taken near \(1\). If
\(|u|\leq4\log x/x\), then, for all sufficiently large \(x\),
\[
  \left|\frac{1+\log x+i\pi+u}{x}\right|
  \leq\frac{2\log x}{x}.
\]
Using \(|\Log(1-v)|\leq2|v|\) for \(|v|\leq1/2\), the right-hand side
maps the disk \(|u|\leq4\log x/x\) into itself. Its derivative with respect
to \(u\) is
\[
  -\frac1{x-1-\log x-i\pi-u}
  =O\left(\frac1x\right),
\]
so it is a contraction on this disk. This proves existence and
uniqueness. Expanding the logarithm gives
\[
  u
  =-\frac{1+\log x+i\pi}{x}
   +O\!\left(\Bigl(\frac{\log x}{x}\Bigr)^2\right),
\]
which proves \eqref{eq:xi-expansion}.

Differentiating \eqref{eq:saddle-equation-xi} gives
\[
  \e^{\xi(x)}\xi'(x)=\xi'(x)-1.
\]
Since \(\e^{\xi(x)}=1+\xi(x)-x\), it follows that
\[
  \bigl(x-\xi(x)\bigr)\xi'(x)=1,
\]
which proves \eqref{eq:xi-derivative}.
\end{proof}

For an integer \(k\) sufficiently large, put
\[
  \xi_k:=\xi(k),
  \qquad
  z_k:=-\xi_k,
  \qquad
  r_k:=|z_k|,
  \qquad
  B_k:=\xi_k(k-\xi_k).
\]
We take
\[
  \vartheta_k:=\arg z_k\in\Bigl(-\pi,-\frac{\pi}{2}\Bigr),
\]
and throughout \(\sqrt{B_k}\) denotes the principal square root, characterized by
\((\sqrt{B_k})^2=B_k\) and \(\Re\sqrt{B_k}>0\).

Lemma~\ref{lem:saddle-points} gives
\begin{align}
  x_k
  &:=\Re{\xi_k} = 
  \log k+O\left(\frac{\log k}{k}\right),
  &
  y_k
  &:= \Im{\xi_k} =
  \pi+O\left(\frac{\log k}{k}\right),
  \label{eq:saddle-coordinate-estimates}\\
  \vartheta_k
  &=
  -\pi+\frac{\pi}{r_k}
  +O\left(\frac1{r_k^2}\right),
  &
  r_k
  &=
  \log k+O\left(\frac1{\log k}\right).
  \label{eq:saddle-angle-estimate}
\end{align}
The saddle equation also gives
\begin{equation}\label{eq:exp-xk-estimate}
  \e^{x_k}=k+O(\log k)
\end{equation}
and
\begin{equation}\label{eq:Bk-real-size}
  \Re B_k\asymp |B_k| =(1+o(1)) k\log k.
\end{equation}
Define
\begin{equation}\label{eq:Qk-Ak-def}
  Q_k:=
  \frac{f(z_k)}{z_k^k\sqrt{2\pi B_k}},
  \qquad
  \theta_k:=\arg Q_k.
\end{equation}

As is standard in the saddle-point method, we will expand \(H(z)\) in a neighbourhood of \(z_k\) to approximate the integrand with a Gaussian. The following lemma gives the necessary Taylor expansion and bounds on the remainder term.
\begin{lemma}
\label{lem:local-saddle-taylor}
Put 
\[
  F_k(t):=H(z_k\e^{it})-ikt.
\]
There exist absolute constants \(\eta,C>0\) such that, uniformly for
\(
  |t|\leq\eta/\log k,
\)
one has
\begin{equation}\label{eq:local-saddle-expansion}
  F_k(t)
  =H(z_k)-\frac12B_kt^2+R_k(t),
  \qquad
  |R_k(t)|\leq Ck(\log k)^2|t|^3.
\end{equation}
\end{lemma}

\begin{proof}
Set
\[
  \Psi(z):=zH'(z)=1-z-\e^{-z} \quad\text{ and }\quad  z(t):=z_k\e^{it}.
\]
Since \(z_k\) is a saddle,
\[
  \Psi(z_k)=k.
\]
By the chain rule,
\[
  F_k'(t)
  =iz(t)H'(z(t))-ik
  =i\{\Psi(z(t))-k\}.
\]
Thus
\[
  F_k'(0)=0.
\]
Differentiating once more gives
\[
  F_k''(t)=-z(t)\Psi'(z(t)).
\]

The saddle equation gives
\[
  \e^{-z_k}=\e^{\xi_k}=1+\xi_k-k,
\]
and hence using
\[
  \Psi'(z)=-1+\e^{-z},
\]
we get
\begin{align*}
  F_k''(0)=-z_k\Psi'(z_k)=\xi_k\bigl(-1+\e^{\xi_k}\bigr)=-\xi_k(k-\xi_k)=-B_k.
\end{align*}

A further differentiation yields
\begin{equation}\label{eq:F-third-derivative}
  F_k'''(t)
  =-i\left\{
    z(t)\Psi'(z(t))+z(t)^2\Psi''(z(t))
  \right\},
\end{equation}
where
\[
  \Psi''(z)=-\e^{-z}.
\]
% By Lemma~\ref{lem:saddle-points},
% \[
%   |z_k|\asymp \log k,
%   \qquad
%   \e^{-\Re z_k}=\e^{\Re\xi_k}\asymp k.
% \]
If \(|t|\leq\eta/\log k\), then
\[
  \Re(-z_k\e^{it})
  =\Re\xi_k+O(|\xi_k||t|)
  =\Re\xi_k+O(\eta).
\]
Consequently,
\[
  \e^{-\Re z(t)}\ll k
\]
uniformly in this range. Since \(|z(t)|=r_k\asymp \log k\),
\eqref{eq:F-third-derivative} gives
\[
  |F_k'''(t)|
  \ll
  |z(t)|\bigl(1+\e^{-\Re z(t)}\bigr)
  +|z(t)|^2\e^{-\Re z(t)}
  \ll k(\log k)^2.
\]
Taylor's formula with integral remainder now gives
\[
  F_k(t)
  =F_k(0)+F_k'(0)t+\frac12F_k''(0)t^2
   +\frac12\int_0^t(t-u)^2F_k'''(u)\dd u.
\]
Using \(F_k(0)=H(z_k)\), \(F_k'(0)=0\), and
\(F_k''(0)=-B_k\), we obtain
\eqref{eq:local-saddle-expansion}.
\end{proof}
% The decomposition into neighborhoods of the two conjugate saddle
% points and the remaining arcs is standard in saddle-point arguments
% for exponential-integral phases; compare
% Hildebrand~\cite{Hildebrand1990}, Hildebrand and
% Tenenbaum~\cite{HildebrandTenenbaum1993}, and
% Weingartner~\cite{Weingartner2026}. We include the estimates needed in
% our setting.
The following lemma provides a technical input that helps us simplify the analysis.
\begin{lemma}
\label{lem:first-angular-maximum}
For
\(
  -\pi\leq\vartheta\leq-\frac{\pi}{2},
\)
put
\[
\begin{split}
  \mathcal M_k(\vartheta)
  :=
  \frac{\e^{-r_k\cos\vartheta}}{r_k}
  \cos\bigl(r_k\sin\vartheta+\vartheta\bigr)-
  \frac{\e^{-r_k\cos\vartheta}}{r_k^2}
  \cos\bigl(r_k\sin\vartheta+2\vartheta\bigr).
\end{split}
\]
For every fixed \(\eta_0>0\), there exists \(c_{\eta_0}>0\) such
that, for all sufficiently large \(k\),
\begin{equation}\label{eq:first-angular-maximum}
  \mathcal M_k(\vartheta)
  \leq
  \frac{\e^{x_k}}{r_k}
  \left(1+\frac1{r_k}\right)
  -
  c_{\eta_0}\frac{\e^{x_k}}{r_k^2},
\end{equation}
provided that
\[
  r_k|\sin\vartheta|\geq1\quad\text{and}\quad  |\vartheta-\vartheta_k|
  \geq\frac{\eta_0}{r_k}.
\]
\end{lemma}

\begin{proof}
Put
\[
  u:=\pi+\vartheta,
  \qquad
  u_k:=\pi+\vartheta_k.
\]
Thus \(0\leq u\leq\pi/2\), and
\[
  u_k=\frac{\pi}{r_k}+O\left(\frac1{r_k^2}\right).
\]
Also, putting
\[
  q_k(u):=r_k\sin u-u,
\]
we have
\begin{equation}\label{eq:Mk-u-form}
  \mathcal M_k(\vartheta)
  =
  \frac{\e^{r_k\cos u}}{r_k}
  \left\{
    -\cos q_k(u)
    -
    \frac1{r_k}\cos\bigl(q_k(u)-u\bigr)
  \right\}.
\end{equation}

Suppose first that \(0\leq u\leq2u_k\). In this range, 
\[
  q_k'(u)=r_k\cos u-1\asymp r_k.
\]
Moreover,
\[
  q_k(u_k)
  =
  r_k\sin u_k-u_k
  =
  \pi+O\left(\frac1{r_k}\right)
\]
and 
\[
q_k(2u_k)=r_k\sin(2u_k)-2u_k=2y_k\cos u_k-2u_k=2\pi-\frac{2\pi}{r_k}
    +O\left(\frac1{r_k^2}\right)
  <2\pi.
\]
The hypothesis
\[
  |u-u_k|
  =
  |\vartheta-\vartheta_k|
  \geq\frac{\eta_0}{r_k}
\]
and the mean-value theorem therefore give
\[
  |q_k(u)-\pi|\geq c\eta_0
\]
for all sufficiently large \(k\). Hence, as $q_k(u)\in [0, 2\pi)$ for $u\in [0, 2u_k]$, we have
\[
  -\cos q_k(u)\leq1-c_0
\]
for some \(c_0=c_0(\eta_0)>0\). Since
\[
  \e^{r_k\cos u}
  \leq\e^{r_k}
  =
  \e^{x_k}
  \left(1+O\left(\frac1{r_k}\right)\right),
\]
equation \eqref{eq:Mk-u-form} gives
\[
  \mathcal M_k(\vartheta)
  \leq
  (1-c_0/2)\frac{\e^{x_k}}{r_k},
\]
which is stronger than \eqref{eq:first-angular-maximum}.

Suppose now that \(u\geq2u_k\). One has
\begin{align*}
  x_k-r_k\cos(2u_k)=
  r_k\bigl(\cos u_k-\cos(2u_k)\bigr)=
  \frac{3\pi^2}{2r_k}
  +O\left(\frac1{r_k^2}\right).
\end{align*}
In particular, as cosine is decreasing, we have
\[
  r_k\cos u\leq x_k-\frac{10}{r_k}
\]
for all sufficiently large \(k\). Hence,
using \(-\cos v\leq1\) in \eqref{eq:Mk-u-form}, we obtain
\begin{align*}
  \mathcal M_k(\vartheta)
  \leq
  \frac{\e^{r_k\cos u}}{r_k}
  \left(1+\frac1{r_k}\right)\leq
  \frac{\e^{x_k}}{r_k}
  \left(1+\frac1{r_k}\right)
  -
  7\frac{\e^{x_k}}{r_k^2}.
\end{align*}
This proves the lemma.
\end{proof}
The following lemma essentially shows that $|f(z)|$ can be large only in a small neighborhood of the saddle points \(z_k\) and \(\overline{z_k}\) on the circle $\{z: |z|=r_k\}$.
\begin{lemma}\label{lem:saddle-separation}
Define
\[
  d_k(\vartheta)
  :=\min_{m\in\mathbb Z}
  \min\left\{
    |\vartheta-\vartheta_k+2\pi m|,
    |\vartheta+\vartheta_k+2\pi m|
  \right\}.
\]
There exist absolute constants \(c,\eta>0\) such that, for all
sufficiently large \(k\),
\begin{align}
  \Re H(r_k\e^{i\vartheta})
  &\leq
  \Re H(z_k)-ck\log k\,d_k(\vartheta)^2,
  &&d_k(\vartheta)\leq\frac{\eta}{\log k},
  \label{eq:separation-near}\\
  \Re H(r_k\e^{i\vartheta})
  &\leq
  \Re H(z_k)-c\frac{k}{(\log k)^2},
  &&d_k(\vartheta)\geq\frac{\eta}{\log k}.
  \label{eq:separation-far}
\end{align}
\end{lemma}
\begin{proof}
We begin near \(z_k\). Write
\[
  \vartheta=\vartheta_k+t,
  \qquad
  |t|\leq\frac{\eta}{\log k}.
\]
Since \(-ikt\) is purely imaginary,
Lemma~\ref{lem:local-saddle-taylor} gives
\[
  \Re H(z_k\e^{it})
  =
  \Re H(z_k)
  -\frac12(\Re B_k)t^2
  +O\bigl(k(\log k)^2|t|^3\bigr).
\]
By \eqref{eq:Bk-real-size}, the ratio of the error term to
\((\Re B_k)t^2\) is \(O((\log k)|t|)\). After decreasing \(\eta\), if
necessary, this ratio is at most \(1/4\). Therefore
\[
  \Re H(z_k\e^{it})
  \leq
  \Re H(z_k)-c_1k\log k\,t^2
\]
for some absolute constant \(c_1>0\). The same argument at \(\overline{z_k}\) proves
\eqref{eq:separation-near}.

We now consider the remainder of the circle. Put
\(
  z:=r_k\e^{i\vartheta}\). Since \(H\) has real
Taylor coefficients, it is enough to consider
\[
  -\pi\leq\vartheta\leq0.
\]
On the left semicircle, assume first that
\[
  -\pi\leq\vartheta\leq-\frac{\pi}{2},
  \qquad
  r_k|\sin\vartheta|\geq1.
\]
By Lemma~\ref{lem:E1-formulas}, 
we have
\[
  E_1(z)
  =
  \e^{-z}
  \left(
    \frac1z-\frac1{z^2}+\frac2{z^3}
  \right)
  +R_3(z),
\]
where
\[
  |R_3(z)|
  \leq
  \frac{6\e^{-\Re z}}
  {r_k^4|\sin\vartheta|}\leq \frac{6\e^{-\Re z}}{r_k^3}.
\]
Therefore
\[
  E_1(z)
  =
  \e^{-z}
  \left(
    \frac1z-\frac1{z^2}+\frac2{z^3}
  \right)
  +
  O\left(
    \frac{\e^{-\Re z}}{r_k^3}
  \right).
\]
Substituting \(z=r_k\e^{i\vartheta}\) and using Lemma~\ref{lem:E1-formulas},
we have \(H(z)=-z+E_1(z)+\Log z+\gamma\). Taking real parts, the
first two terms in the expansion of \(E_1\) contribute \(\mathcal M_k(\vartheta)\),
and \(\Re\Log(r_k\e^{i\vartheta})=\log r_k\).
\begin{equation}\label{eq:real-H-angular}
  \Re H(r_k\e^{i\vartheta})
  =
  -r_k\cos\vartheta+\log r_k+\gamma
  +\mathcal M_k(\vartheta)
  +
  O\left(
    \frac{\e^{-r_k\cos\vartheta}}{r_k^3}
  \right).
\end{equation}

Taking \(\vartheta=\vartheta_k\) in the same formula, note that
\(r_k\cos\vartheta_k=-x_k\) and \(r_k\sin\vartheta_k=-y_k\). Using
\eqref{eq:saddle-coordinate-estimates}--\eqref{eq:saddle-angle-estimate},
\(\mathcal M_k(\vartheta_k)\) gives the displayed main term, while all
other terms in \eqref{eq:real-H-angular} are lower order.
\begin{equation}\label{eq:H-at-saddle-real}
  \Re H(z_k)
  =
  \frac{\e^{x_k}}{r_k}
  \left(
    1+\frac1{r_k}
    +O\left(\frac1{r_k^2}\right)
  \right).
\end{equation}
Suppose that
\[
  d_k(\vartheta)\geq\frac{\eta}{\log k}.
\]
Since \(r_k\asymp\log k\), this implies
\[
  |\vartheta-\vartheta_k|
  \geq\frac{\eta_0}{r_k}
\]
for some fixed \(\eta_0>0\). Lemma~\ref{lem:first-angular-maximum}
and \eqref{eq:real-H-angular} therefore give, for some \(c_2>0\),
\[
\begin{split}
  \Re H(r_k\e^{i\vartheta})
  \leq
  O(r_k)
  +
  \frac{\e^{x_k}}{r_k}
  \left(1+\frac1{r_k}\right)-
  c_2\frac{\e^{x_k}}{r_k^2}
  +
  O\left(\frac{\e^{x_k}}{r_k^3}\right).
\end{split}
\]
Because
\[
  r_k=o\left(\frac{\e^{x_k}}{r_k^2}\right),
\]
comparison with \eqref{eq:H-at-saddle-real} and substituting~\eqref{eq:exp-xk-estimate}
yields, for some \(c_3>0\),
\begin{equation*}\label{eq:left-arc-separation}
  \Re H(r_k\e^{i\vartheta})
  \leq
  \Re H(z_k)-c_3\frac{k}{r_k^2}.
\end{equation*}

There are only two elementary pieces not covered by this argument.
On \(
  -\frac{\pi}{2}\leq\vartheta\leq0,
\)
\[
  \Re H(r_k\e^{i\vartheta}) = -r_k\cos\vartheta+\int_{0}^1\frac{\Re(1-\e^{-zt})}{t}\dd t
  \leq\int_0^1\frac{|1-\e^{-zt}|}{t}\dd t.
\]
For \(t\leq1/r_k\), the numerator is at most \(r_kt\), while for
\(t\geq1/r_k\) it is at most \(2\); hence the integral is at most
\(1+2\log r_k\). Since \(r_k\asymp\log k\), we have
\(1+\log r_k=o(k/r_k^2)\).

Finally, suppose that
\[
  -\pi\leq\vartheta\leq-\frac{\pi}{2},
  \qquad r_k|\sin\vartheta|<1.
\]
The integral representation gives
\[
\Re H(r_k\e^{i\vartheta})
  =-r_k\cos\vartheta
  +\int_0^1
  \frac{1-\e^{-r_kt\cos\vartheta}
  \cos(r_kt\sin\vartheta)}{t}\dd t.
\]
For \(0\leq t\leq1\), using \(\log\cos u\geq-u^2\) for
\(|u|\leq1\), we have
\[
  -r_kt\cos\vartheta+\log\cos(r_kt\sin\vartheta)
  \geq
  t\bigl(\sqrt{r_k^2-1}-1\bigr)\geq0.
\]
Thus the integrand is nonpositive, and restricting the integral to
\(1-1/r_k\leq t\leq1\) gives
\[
  \Re H(r_k\e^{i\vartheta})
  \leq
  -r_k\cos\vartheta
  +\int_{1-1/r_k}^1
  \frac{1-\e^{-r_kt\cos\vartheta}
  \cos(r_kt\sin\vartheta)}{t}\dd t
  \leq
  -r_k\cos\vartheta
  -c_4\frac{\e^{-r_k\cos\vartheta}}{r_k}
\]
for some absolute constant \(c_4>0\). Since \(r_k|\sin\vartheta|<1<y_k\), we have
\(-r_k\cos\vartheta>x_k\); hence the last expression is
\[
  \leq r_k-c_4\frac{\e^{x_k}}{r_k}
  \leq-c_5\frac{k}{r_k}\leq -c_5\frac{k}{(\log k)^2}
\]
for some absolute constant \(c_5>0\).
\end{proof}
\begin{proposition}
\label{prop:f-saddle}
As \(k\to\infty\),
\begin{equation}\label{eq:f-coefficient-asymptotic}
  \frac{f^{(k)}(0)}{k!}
  =2|Q_k|
  \left(
    \cos\theta_k
    +O\left(\sqrt{\frac{\log k}{k}}\right)
  \right).
\end{equation}
Additionally, as \(k\to\infty\),
\begin{equation}\label{eq:Qk-size}
  |Q_k|
  =
  \frac{1}{(\log k)^k\sqrt{2\pi k\log k}}
  \exp\left\{
    \frac{k}{\log k}
    +O\left(\frac{k}{(\log k)^2}\right)
  \right\}.
\end{equation}
\end{proposition}

\begin{proof}
Cauchy's formula on \(|z|=r_k\) gives
\begin{equation}\label{eq:cauchy-saddle-integral}
  \frac{f^{(k)}(0)}{k!}
  =\frac1{2\pi r_k^k}
  \int_{-\pi}^{\pi}
  \exp\{H(r_k\e^{i\vartheta})-ik\vartheta\}
  \dd\vartheta.
\end{equation}
Define
\begin{equation*}\label{eq:delta-k-def}
  \sigma_k:=\frac{k^{-2/5}}{\log k}.
\end{equation*}
Any
choice
\[
  \sigma_k=\frac{k^{-a}}{\log k},
  \qquad \frac13<a<\frac12,
\]
would serve equally well in what follows; the exponent \(2/5\) is chosen only for
definiteness.

We first consider
\[
  |\vartheta-\vartheta_k|\leq\sigma_k.
\]
Writing \(\vartheta=\vartheta_k+t\),
Lemma~\ref{lem:local-saddle-taylor} gives
\[
  H(z_k\e^{it})-ikt
  =H(z_k)-\frac12B_kt^2+R_k(t).
\]
Since
\[
  k(\log k)^2\sigma_k^3
  =\frac{k^{-1/5}}{\log k}
  \longrightarrow0,
\]
we have
\[
  \sup_{|t|\leq\sigma_k}|R_k(t)|=o(1).
\]
Hence \(|\e^{R_k(t)}-1|\ll|R_k(t)|\) uniformly on this interval, and
\begin{align*}
  \int_{-\sigma_k}^{\sigma_k}
  \exp\left\{-\frac12B_kt^2+R_k(t)\right\}\dd t=
  \int_{-\sigma_k}^{\sigma_k}\e^{-B_kt^2/2}\dd t
  +O\left(
    k(\log k)^2
    \int_{-\sigma_k}^{\sigma_k}
    |t|^3\e^{-c k\log k\,t^2}\dd t
  \right).
\end{align*}
Additionally, since
\[
  \sigma_k\sqrt{\Re B_k}
  \asymp
  \frac{k^{1/10}}{\sqrt{\log k}}
  \longrightarrow\infty,
\]
the standard Gaussian tail bound gives
\[
  \int_{|t|>\sigma_k}\e^{-(\Re B_k)t^2/2}\dd t
  \ll
  \frac{\exp\{-(\Re B_k)\sigma_k^2/4\}}{\sqrt{\Re B_k}}
  =
  o\left(
    \frac1{\sqrt{|B_k|}}\sqrt{\frac{\log k}{k}}
  \right).
\]

Moreover,
\[
  \int_{-\infty}^{\infty}
  |t|^3\e^{-c k\log k\,t^2}\dd t
  \ll\frac1{k^2(\log k)^2}.
\]
 Since
\(\Re B_k>0\),
\[
  \int_{-\infty}^{\infty}\e^{-B_kt^2/2}\dd t
  =\sqrt{\frac{2\pi}{B_k}},
\]
with the principal square root. Consequently,
\[
  \int_{-\sigma_k}^{\sigma_k}
  \exp\left\{-\frac12B_kt^2+R_k(t)\right\}\dd t
  =\sqrt{\frac{2\pi}{B_k}}
   \left(
     1+O\left(\sqrt{\frac{\log k}{k}}\right)
   \right).
\]

Since \(z_k=r_k\e^{i\vartheta_k}\), the contribution of this arc to
\eqref{eq:cauchy-saddle-integral} is
\[
  Q_k
  \left(
    1+O\left(\sqrt{\frac{\log k}{k}}\right)
  \right).
\]
The conjugate arc through \(\overline{z_k}\) contributes
\[
  \overline{Q_k}
  \left(
    1+O\left(\sqrt{\frac{\log k}{k}}\right)
  \right).
\]

Using \eqref{eq:separation-near} if \(
  \sigma_k\leq d_k(\vartheta)
  \leq\eta/\log k,
\) and~\eqref{eq:separation-far} if 
\(
  d_k(\vartheta)\geq\frac{\eta}{\log k}
\)
shows that the contribution from the complement of the two saddle arcs is
\[
  o\left(
    |Q_k|\sqrt{\frac{\log k}{k}}
  \right).
\]
Adding the two saddle contributions gives
\[
  \frac{f^{(k)}(0)}{k!}
  =Q_k+\overline{Q_k}
   +O\left(
     |Q_k|\sqrt{\frac{\log k}{k}}
   \right).
\]
Since
\[
  Q_k+\overline{Q_k}=2|Q_k|\cos\theta_k,
\]
we obtain \eqref{eq:f-coefficient-asymptotic}. 

By the definition of \(Q_k\),
\[
  |Q_k|
  =
  \frac{\exp\{\Re H(z_k)\}}
       {r_k^k\sqrt{2\pi|B_k|}}.
\]
Substituting \eqref{eq:H-at-saddle-real},  \eqref{eq:exp-xk-estimate}, \eqref{eq:saddle-coordinate-estimates}, and
\eqref{eq:Bk-real-size} into the definition of \(|Q_k|\) gives
\eqref{eq:Qk-size}.
\end{proof}

It remains to record the slow variation of the phase. For real \(x\)
sufficiently large, put
\[
  z(x):=-\xi(x),
  \qquad
  \mathcal B(x):=\xi(x)(x-\xi(x)),
\]
and define continuously
\begin{equation}\label{eq:continuous-phase}
  \Theta_0(x)
  :=\Im\left\{
    H(z(x))-x\Log z(x)-\frac12\Log(2\pi \mathcal B(x))
  \right\} -\pi x.
\end{equation}
Thus
\[
  \Theta_0(k)\equiv\theta_k-\pi k\pmod{2\pi}.
\]

\begin{lemma}
\label{lem:phase-variation}
One has
\begin{equation}\label{eq:phase-derivative}
  \Theta_0'(x)
  =-\frac{\pi}{\log x}
   +O\!\left(
     \frac1{(\log x)^3}+\frac1x
   \right).
\end{equation}
Consequently, there exist sequences of integers
\[
  k_j^+,k_j^-\longrightarrow\infty
\]
such that
\begin{equation}\label{eq:phase-good-sequences}
  (-1)^{k_j^+-1}\cos\theta_{k_j^+}\longrightarrow1,
  \qquad
  (-1)^{k_j^--1}\cos\theta_{k_j^-}\longrightarrow-1.
\end{equation}
\end{lemma}

\begin{proof}
Differentiate \eqref{eq:continuous-phase}. The saddle equation~\eqref{eq:saddle-equation-xi}
\[
  z(x)H'(z(x))=1-z(x)-e^{-z(x)} = x
\]
implies
\[
  H'(z(x))z'(x)-x\frac{z'(x)}{z(x)}=0.
\]
Therefore
\[
  \frac{\dd}{\dd x}
  \left\{
    H(z(x))-x\Log z(x)-\frac12\Log(2\pi \mathcal B(x))
  \right\}
  =-\Log z(x)-\frac12\frac{\mathcal B'(x)}{\mathcal B(x)}.
\]
Taking imaginary parts gives
\[
  \Theta_0'(x)
  =-\arg z(x)-\frac12\Im\frac{\mathcal B'(x)}{\mathcal B(x)}-\pi.
\]
By \eqref{eq:xi-expansion},
\[
  z(x)
  =-\log x-i\pi
   +O\left(\frac{\log x}{x}\right).
\]
Consequently,
\[
  \arg z(x)
  =-\pi+\arctan\left(\frac{\pi}{\log x}\right)
   +O\left(\frac1x\right)
\]
and hence
\[
  \arg z(x)
  =-\pi+\frac{\pi}{\log x}
   +O\!\left(
     \frac1{(\log x)^3}+\frac1x
   \right).
\]
Furthermore, using
\[
  \mathcal B(x)=\xi(x)(x-\xi(x))
\]
and \eqref{eq:xi-derivative}, we obtain
\[
  \frac{\mathcal B'(x)}{\mathcal B(x)}
  =\frac{\xi'(x)}{\xi(x)}
   +\frac{1-\xi'(x)}{x-\xi(x)}
  =O\left(\frac1x\right).
\]
It follows that
\[
  \Theta_0'(x)
  =-\frac{\pi}{\log x}
   +O\!\left(
     \frac1{(\log x)^3}+\frac1x
   \right),
\]
which proves \eqref{eq:phase-derivative}.

Thus \(\Theta_0\) is eventually strictly decreasing,
\[
  \Theta_0(x)\longrightarrow-\infty,
\]
and
\[
  \Theta_0(x+1)-\Theta_0(x)\longrightarrow0.
\]
Choose integers \(k_j^+\) nearest to successive solutions of
\[
  \Theta_0(x)=(2m+1)\pi,
\]
and integers \(k_j^-\) nearest to successive solutions of
\[
  \Theta_0(x)=2m\pi.
\]
Because the increment of \(\Theta_0\) over an interval of length one
tends to zero, we have
\[
  \cos\Theta_0(k_j^+)\longrightarrow-1,
  \qquad
  \cos\Theta_0(k_j^-)\longrightarrow1.
\]
Finally,
\[
  \theta_k\equiv\Theta_0(k)+\pi k\pmod{2\pi},
\]
so
\[
  \cos\theta_k=(-1)^k\cos\Theta_0(k).
\]
This proves \eqref{eq:phase-good-sequences}.
\end{proof}
\section{Lower Bound on \texorpdfstring{\(D_k\)}{Dk}}\label{sec:lower-bound-Dk}
In this section, we prove lower bounds on \(D_k\) along appropriate subsequences by replacing $G_p(z)$ with $f(z)$. But before doing that, we need to establish some estimates for $G_p(z)$ to justify the replacement.

Throughout this section, \(p\) and \(q\) denote primes.  The normalized
function \(G_p\) in \eqref{eq:normalized-Gp} has the exact
factorization
\begin{align}
  G_p(z)&=D_p(z)P_p(z),
    \label{eq:G-factorization}\\
  D_p(z)&=\e^{-z}\left(1+p(1-\e^{-\delta_pz})\right),
  \qquad
  \delta_p:=\frac{\log(1+1/p)}{\mathcal L_p},
    \label{eq:D-factor}\\
  P_p(z)&=\prod_{q<p}
  \frac{1-q^{-1}\e^{-z\log q/\mathcal L_p}}{1-q^{-1}}.
    \label{eq:P-factor}
\end{align}
Both factors are nonvanishing on \(|z|<\mathcal L_p\).  Indeed, their
possible zeros have real part \(-\mathcal L_p\).  On this disk we write
\(\Log D_p\) and \(\Log P_p\) for the unique holomorphic logarithms
normalized by
\[
  \Log D_p(0)=\Log P_p(0)=0.
\]
Expanding the Euler factors gives
\begin{equation}\label{eq:log-Pp-expansion}
  \Log P_p(z)
  =\sum_{q<p}\sum_{m\geq1}
  \frac{q^{-m}}m
  \left(1-\e^{-mz\log q/\mathcal L_p}\right).
\end{equation}

We first collect the estimates for \(D_p\) and \(P_p\) that will be
used both in the proof of the approximation of $G_p$ by \(f\) and in the treatment of the
smaller primes.

\begin{lemma}
\label{lem:common-factor-estimates}
There is an absolute constant \(C>0\) such that, for
\(0<\rho\leq\mathcal L_p/2\),
\begin{align}
  \max_{|z|\leq\rho}|D_p(z)|
  &\leq2\e^\rho,
    \label{eq:Dp-coarse-bound}\\
  \max_{|z|\leq\rho}|P_p(z)|
  &\leq
  \exp\left\{
    C\left(\frac{\e^\rho}{\rho}+1+\mathcal L_p\right)
  \right\}.
    \label{eq:Pp-coarse-bound}
\end{align}
Moreover, uniformly for \(|z|\leq\mathcal L_p/4\),
\begin{align}
  \Log P_p(z)
  &=\sum_{q<p}\frac{1-\e^{-z\log q/\mathcal L_p}}q
    +O\!\left(\frac{|z|}{\mathcal L_p}\right),
    \label{eq:Pp-first-order-decomposition}\\
  \Log D_p(z)+z
  &=\log\left(1+\frac{z}{\mathcal L_p}\right)
    +O\!\left(\frac{(1+|z|)^2}{p}\right).
    \label{eq:Dp-refined-expansion}
\end{align}
The second formula holds for all sufficiently large \(p\).
\end{lemma}

\begin{proof}
Since \(p\log(1+1/p)\leq1\), for \(|z|\leq\rho\) we have
\[
  p\delta_p|z|\leq\frac{\rho}{\mathcal L_p}\leq\frac12.
\]
Thus \(|1-\e^{-u}|\leq|u|\e^{|u|}\) and
\eqref{eq:D-factor} give \eqref{eq:Dp-coarse-bound}.

For the terms with \(m=1\) in \eqref{eq:log-Pp-expansion}, partial
summation and Chebyshev's estimate give absolute constants \(A,B>0\)
such that
\[
  \sum_{q<p}\frac{\exp(\rho\log q/\mathcal L_p)}q
  \leq A\log(2+\mathcal L_p)+B\frac{\e^\rho}{\rho}.
\]
For \(m\geq2\), the inequality \(\rho/\mathcal L_p\leq1/2\) gives
\[
  q^{-m}\left|\e^{-mz\log q/\mathcal L_p}\right|
  \leq q^{-m/2}.
\]
Together with \(\sum_{q<p}q^{-1}\ll1+\mathcal L_p\), these estimates
prove \eqref{eq:Pp-coarse-bound}.

Separating the term \(m=1\) in \eqref{eq:log-Pp-expansion} and using
\(|1-\e^{-u}|\leq |u|\e^{|u|}\) once again, we obtain, for
\(|z|\leq\mathcal L_p/4\),
\begin{align*}
  &\sum_{q<p}\sum_{m\geq2}\frac{q^{-m}}m
   \left|1-\e^{-mz\log q/\mathcal L_p}\right|\\
  &\qquad\ll
  \frac{|z|}{\mathcal L_p}
  \sum_q(\log q)\sum_{m\geq2}q^{-3m/4}
  \ll\frac{|z|}{\mathcal L_p},
\end{align*}
which proves \eqref{eq:Pp-first-order-decomposition}.  Finally, observe that
\[
  p\delta_p
  =\frac1{\mathcal L_p}
   +O\!\left(\frac1{p\mathcal L_p}\right).
\]
Expanding $e^{-\delta_pz}$ in \eqref{eq:D-factor} using Taylor's formula gives
\[
  1+p(1-\e^{-\delta_pz})
  =1+\frac{z}{\mathcal L_p}
   +O\!\left(\frac{(1+|z|)^2}{p}\right).
\]
Since \(|1+z/\mathcal L_p|\geq3/4\), taking the $\Log$
proves \eqref{eq:Dp-refined-expansion}.
\end{proof}
We can immediately deduce a Cauchy estimate for the derivatives of \(G_p\) at the origin.
\begin{lemma}
\label{lem:general-circle-bound}
There is an absolute constant \(C_*>0\) such that, for every prime
\(p\), every \(j\geq0\), and every
\(0<\rho\leq\mathcal L_p/2\),
\begin{equation*}\label{eq:general-circle-bound}
  |G_p^{(j)}(0)|
  \leq\frac{2j!}{\rho^j}
  \exp\left\{
    C_*\left(
      \rho+\frac{\e^\rho}{\rho}+1+\mathcal L_p
    \right)
  \right\}.
\end{equation*}
\end{lemma}

\begin{proof}
Multiply \eqref{eq:Dp-coarse-bound} and
\eqref{eq:Pp-coarse-bound}, and apply Cauchy's estimate on
\(|z|=\rho\).
\end{proof}

We next compare \(G_p\) with the limiting function \(f\).

\begin{lemma}
\label{lem:uniform-Gp}
For every fixed \(C_0>0\), there are constants \(A_0,L_0>0\) such
that, if \(\mathcal L_p\geq L_0\), there is a function \(\Delta_p\),
holomorphic on a neighborhood of
\(|z|\leq C_0\log\mathcal L_p\), satisfying
\[
  \exp\Delta_p(z)=\frac{G_p(z)}{f(z)},
  \qquad \Delta_p(0)=0,
\]
and
\begin{equation}\label{eq:Delta-p-bound}
  |\Delta_p(z)|
  \leq A_0\frac{\log\mathcal L_p}{\mathcal L_p}.
\end{equation}
\end{lemma}

\begin{proof}
For sufficiently large \(p\), the disk in the statement lies in
\(|z|\leq\mathcal L_p/4\), and we may put
\begin{align*}
  \Delta_p(z)
  :=\Log D_p(z)+\Log P_p(z)-H(z)\\
  =\Log D_p(z)+\Log P_p(z)-\Ein(z)+z.
\end{align*}
In view of Lemma~\ref{lem:common-factor-estimates}, it suffices to estimate the term
\[
  \sum_{q<p}\frac{1-\e^{-z\log q/\mathcal L_p}}q
\]
in \eqref{eq:log-Pp-expansion}.  Define
\[
  \mathcal P(y):=\sum_{\log q\leq y}\frac{\log q}{q},
  \qquad
  \vartheta(x):=\sum_{q\leq x}\log q.
\]
Partial summation gives
\[
  \mathcal P(y)
  =\e^{-y}\vartheta(\e^y)
   +\int_2^{\e^y}\frac{\vartheta(t)}{t^2}\dd t.
\]
The quantitative prime number theorem (see \cite[Theorem~6.9, equation~(6.13)]{MontgomeryVaughan2007}) implies that, for some
\(\kappa>0\),
\begin{equation}\label{eq:weighted-PNT}
  \mathcal P(y)=y+C_1+\mathcal Q(y),
  \qquad
  \mathcal Q(y)
  =O\bigl((1+y)\e^{-\kappa\sqrt y}\bigr).
\end{equation}
Put
\[
  \varphi_z(t):=
  \begin{cases}
    (1-\e^{-zt})/t,&t\neq0,\\
    z,&t=0.
  \end{cases}
\]
Stieltjes integration gives
\begin{align*}
  \sum_{q<p}\frac{1-\e^{-z\log q/\mathcal L_p}}q
  &\ =\frac1{\mathcal L_p}
  \int_{0-}^{\mathcal L_p-}
  \varphi_z(y/\mathcal L_p)\dd\mathcal P(y)\\
  &=\Ein(z)+\frac{C_1z}{\mathcal L_p}
   +\frac{\varphi_z(1)\mathcal Q(\mathcal L_p-)}{\mathcal L_p}
   -\frac1{\mathcal L_p^2}\int_0^{\mathcal L_p}
    \varphi_z'(y/\mathcal L_p)\mathcal Q(y)\dd y.
\end{align*}
Since
\[
  |\varphi_z'(t)|\leq\frac{|z|^2}{2}\e^{|z|t}
  \qquad(0\leq t\leq1),
\]
\eqref{eq:weighted-PNT} yields, uniformly for
\(|z|\leq C_0\log\mathcal L_p\),
\begin{equation}\label{eq:first-order-Ein-approximation}
  \sum_{q<p}\frac{1-\e^{-z\log q/\mathcal L_p}}q
  =\Ein(z)
   +O_{C_0}\!\left(\frac{\log\mathcal L_p}{\mathcal L_p}\right).
\end{equation}
Combining \eqref{eq:first-order-Ein-approximation} with
\eqref{eq:Pp-first-order-decomposition} and
\eqref{eq:Dp-refined-expansion} proves
\eqref{eq:Delta-p-bound}. 
\end{proof}

As an application of Lemma~\ref{lem:uniform-Gp}, we now get closeness of derivatives of \(G_p\) and \(f\) at the origin.

\begin{lemma}
\label{lem:derivative-approximation}
Uniformly for primes satisfying \(\mathcal L_p\geq {\frac{k}{\log k}}\), \begin{equation*}\label{eq:derivative-approximation} G_p^{(k)}(0)-f^{(k)}(0)=o(2k!|Q_k|). \end{equation*}
\end{lemma}

\begin{proof}
Take \(C_0=2\) in Lemma~\ref{lem:uniform-Gp}. For all sufficiently
large \(k\),
\[
  r_k
  \leq 2\log\frac{k}{\log k}
  \leq 2\log\mathcal L_p.
\]
Hence, uniformly for primes satisfying
\(\mathcal L_p\geq k/\log k\),
\[
  \sup_{|z|=r_k}|\Delta_p(z)|
  \ll
  \frac{\log(k/\log k)}{k/\log k}
  \ll
  \frac{(\log k)^2}{k}
  =o(1).
\]
It follows that, on writing
\[
  \epsilon_k
  :=
  \sup_{\substack{\mathcal L_p\geq k/\log k\\|z|=r_k}}
  \left|\e^{\Delta_p(z)}-1\right|,
\]
we have
\[
  \epsilon_k\ll \frac{(\log k)^2}{k}.
\]

Since \(G_p(z)=f(z)\e^{\Delta_p(z)}\), Cauchy's formula gives
\[
  |G_p^{(k)}(0)-f^{(k)}(0)|
  \leq
  \epsilon_k
  \frac{k!}{r_k^k}
  \max_{|z|=r_k}|f(z)|.
\]
By Lemma~\ref{lem:saddle-separation}, the maximum of
\(\Re H(z)\) on \(|z|=r_k\) is attained at the two conjugate saddle
points. Thus,
\[
  \max_{|z|=r_k}|f(z)|
  \leq |f(z_k)|.
\]
Consequently,
\[
  |G_p^{(k)}(0)-f^{(k)}(0)|
  \leq
  \epsilon_k
  \frac{k!}{r_k^k}
  |f(z_k)|.
\]
Using~\eqref{eq:Qk-Ak-def} and~\eqref{eq:Bk-real-size}, 
we obtain
\[
  |G_p^{(k)}(0)-f^{(k)}(0)|
  \ll
  \frac{(\log k)^{5/2}}{\sqrt{k}}\,
  k!|Q_k|
  =
  o(k!|Q_k|).
\]
\end{proof}

The following lemma allows us to estimate moments of $\mathcal L_p$ weighted by $T_p(1)$.
\begin{lemma}
\label{lem:weighted-prime-moment}
As \(k\to\infty\),
\begin{equation*}\label{eq:weighted-prime-moment}
  \sum_pT_p(1)\mathcal L_p^k
  =\bigl(\e^{-\gamma}+o(1)\bigr)\Gamma(k-1).
\end{equation*}
Additionally,
\begin{equation*}\label{eq:small-weight-tail}
  \sum_{\mathcal L_p<\frac{k}{\log k}}T_p(1)\mathcal L_p^k
  =o\bigl(\Gamma(k-1)\bigr).
\end{equation*}
\end{lemma}

\begin{proof}
Mertens' product theorem gives
\begin{equation}\label{eq:mertens-weight-uniform}
  V(p)\log p=\e^{-\gamma}+o(1).
\end{equation}
We first evaluate
\[
  S_k:=\sum_p\frac{(\log p)^{k-1}}{p^2}.
\]
Writing \(\vartheta(x)=\sum_{p\leq x}\log p\), we have the exact
Stieltjes representation
\[
  S_k
  =\int_{(\log2)^-}^\infty t^{k-2}\e^{-2t}
    \dd\vartheta(\e^t).
\]
The quantitative prime number theorem gives
\[
  \vartheta(\e^t)
  =\e^t+O\bigl(\e^t\e^{-\kappa\sqrt t}\bigr)
\]
for some \(\kappa>0\).  Integration by parts therefore yields
\begin{align*}
  S_k
  =\int_{(\log2)^-}^\infty t^{k-2}\e^{-t}\dd t+O\left(
    \int_{(\log2)^-}^\infty
    t^{k-2}\e^{-t}\e^{-\kappa\sqrt t}
    \left(1+\frac{k}{t}\right)\dd t
  \right).
\end{align*}
The error is \(o(\Gamma(k-1))\): on
\(k/2\leq t\leq2k\) this follows from the factor
\(\e^{-\kappa\sqrt{k/2}}\), while the two complementary ranges are
exponentially small relative to the gamma integral.  Thus
\begin{equation}\label{eq:prime-gamma-moment}
  S_k=(1+o(1))\Gamma(k-1).
\end{equation}
The same argument, or simply
\[
  \int_0^{{\frac{k}{\log k}}}t^{k-2}\e^{-t}\dd t
  \leq\frac{1}{k-1}\Bigl(\frac{k}{\log k}\Bigr)^{k-1},
\]
shows that
\begin{equation}\label{eq:prime-gamma-left-tail}
  \sum_{\mathcal L_p<{\frac{k}{\log k}}}
  \frac{\mathcal L_p^{k-1}}{p^2}
  =o(\Gamma(k-1)).
\end{equation}

Now
\[
  T_p(1)\mathcal L_p^k
  =\frac{V(p)\mathcal L_p^k}{p(p+1)}.
\]
The estimate in \eqref{eq:mertens-weight-uniform} and
\eqref{eq:prime-gamma-left-tail} show that the contribution of
\(\mathcal L_p<{\frac{k}{\log k}}\) is \(o(\Gamma(k-1))\).  On the complementary
range, \({\frac{k}{\log k}}\to\infty\), so the estimate in
\eqref{eq:mertens-weight-uniform} and
\(p/(p+1)=1+o(1)\) hold uniformly.  Hence
\[
  \sum_{\mathcal L_p\geq {\frac{k}{\log k}}}T_p(1)\mathcal L_p^k
  =(\e^{-\gamma}+o(1))
  \sum_{\mathcal L_p\geq {\frac{k}{\log k}}}
  \frac{\mathcal L_p^{k-1}}{p^2}.
\]
Equations \eqref{eq:prime-gamma-moment} and
\eqref{eq:prime-gamma-left-tail} prove both assertions.
\end{proof}

We first show that the contribution of the primes $p$ with \(\mathcal L_p< {\frac{k}{\log k}}\) is negligible in the evaluation of \(D_k\).

\begin{lemma}
\label{lem:smaller-primes}
As \(k\to\infty\),
\begin{equation}\label{eq:smaller-primes}
  \sum_{\mathcal L_p<{\frac{k}{\log k}}}
  T_p(1)\mathcal L_p^k|G_p^{(k)}(0)|
  =o\bigl(\Gamma(k-1)2k!|Q_k|\bigr).
\end{equation}
\end{lemma}

\begin{proof}
Split the sum at \(\mathcal L_p=2r_k\).  For
\(\mathcal L_p<2r_k\), apply
Lemma~\ref{lem:general-circle-bound} with
\(\rho=\mathcal L_p/2\).  Since
\(\e^{r_k}/r_k\ll k/\log k\) using~\eqref{eq:saddle-coordinate-estimates}, this gives
\[
  \sum_{\mathcal L_p<2r_k}
  T_p(1)\mathcal L_p^k|G_p^{(k)}(0)|
  \leq2\Xi\,k!2^k\exp\{O(k/\log k)\}.
\]
For \(2r_k\leq\mathcal L_p<{\frac{k}{\log k}}\), use Lemma~\ref{lem:general-circle-bound} with \(\rho=r_k\).  Then
\[
  \sum_{2r_k\leq\mathcal L_p<{\frac{k}{\log k}}}
  T_p(1)\mathcal L_p^k|G_p^{(k)}(0)|
  \leq
  2\Xi\,k!\left(\frac{{\frac{k}{\log k}}}{r_k}\right)^k
  \exp\{O({\frac{k}{\log k}})\}.
\]
On the other hand, Stirling's formula and~\eqref{eq:Qk-size} give
\begin{equation}\label{eq:total-saddle-scale}
  \log\bigl(\Gamma(k-1)2k!|Q_k|\bigr)
  =2k\log k-k\log\log k-2k
   +O\!\left(\frac{k}{\log k}\right).
\end{equation}
The logarithms of the preceding two bounds are, respectively,
\[
  k\log k+O(k)
  \quad\text{and}\quad
  2k\log k-2k\log\log k+O(k).
\]
Both are smaller than \eqref{eq:total-saddle-scale} by a quantity tending
to infinity faster than \(k\), which proves
\eqref{eq:smaller-primes}.
\end{proof}

We can now evaluate \(D_k\) itself.

\begin{proposition}
\label{prop:Dk-asymptotic}
As \(k\to\infty\),
\begin{equation}\label{eq:Dk-asymptotic}
  D_k
  =\e^{-\gamma}\Gamma(k-1)2k!|Q_k|
   \bigl(\cos\theta_k+o(1)\bigr).
\end{equation}
% In particular,
% \begin{equation}\label{eq:Dk-amplitude-size}
%   \log\bigl(\e^{-\gamma}\Gamma(k-1)2k!|Q_k|\bigr)
%   =k\log\left(\frac{k^2}{\e^2\log k}\right)
%    +O\!\left(\frac{k}{\log k}\right).
% \end{equation}
\end{proposition}

\begin{proof}
Let
\[
  \Sigma_k:=\sum_{\mathcal L_p\geq {\frac{k}{\log k}}}
  T_p(1)\mathcal L_p^k.
\]
By Lemma~\ref{lem:weighted-prime-moment},
\[
  \Sigma_k=(\e^{-\gamma}+o(1))\Gamma(k-1).
\]
Lemma~\ref{lem:derivative-approximation}, the positivity of the weights,
and Proposition~\ref{prop:f-saddle} give
\begin{align*}
  \sum_{\mathcal L_p\geq {\frac{k}{\log k}}}
  T_p(1)\mathcal L_p^kG_p^{(k)}(0)
  &=\Sigma_k f^{(k)}(0)
    +o(\Sigma_k2k!|Q_k|)\\
  &=\e^{-\gamma}\Gamma(k-1)2k!|Q_k|
    \bigl(\cos\theta_k+o(1)\bigr).
\end{align*}
Lemma~\ref{lem:smaller-primes} shows that the complementary prime sum is
absorbed by the error term.  This proves \eqref{eq:Dk-asymptotic}.
%Equation \eqref{eq:Dk-amplitude-size} follows from Stirling's formula.
\end{proof}

\begin{proof}[Proof of Proposition~\ref{prop:Dk-lower}]
Let \(k_j^+\) and \(k_j^-\) be the sequences from
Lemma~\ref{lem:phase-variation}.  By
\eqref{eq:Dk-asymptotic},
\[
  (-1)^{k_j^+-1}D_{k_j^+}
  =\bigl(2+o(1)\bigr)k_j^+!|Q_{k_j^+}|
   \e^{-\gamma}\Gamma(k_j^+-1),
\]
whereas
\[
  (-1)^{k_j^--1}D_{k_j^-}
  =-\bigl(2+o(1)\bigr)k_j^-!|Q_{k_j^-}|
   \e^{-\gamma}\Gamma(k_j^--1).
\]
Stirling's formula and
\eqref{eq:Qk-size} imply, for all sufficiently large \(k\),
\[
  \e^{-\gamma}\Gamma(k-1)2k!|Q_k|
  \geq
  \left(
    \frac1{\e^{2+o(1)}}\frac{k^2}{\log k}
  \right)^k.
\]
The two asserted inequalities follow.
\end{proof}
% \section*{External Analytic Results}

% The quantitative prime number theorem used in
% Lemmas~\ref{lem:uniform-Gp} and \ref{lem:weighted-prime-moment} is
% Montgomery and Vaughan
% \cite[Theorem~6.9, equation~(6.13)]{MontgomeryVaughan2007}.
% The product estimate in Lemma~\ref{lem:weighted-prime-moment} is
% Mertens' third theorem~\cite{Mertens1874}.  The saddle-point argument in
% Proposition~\ref{prop:f-saddle} follows the classical treatment of
% conjugate saddles for exponential-integral phases developed by
% Hildebrand~\cite{Hildebrand1990} and Hildebrand and
% Tenenbaum~\cite{HildebrandTenenbaum1993}; see also the recent explicit
% treatment by Weingartner~\cite{Weingartner2026}.  General versions of
% the saddle-point estimates used there are given by
% Olver~\cite[Chapter~4]{Olver1997}.

% Cauchy's estimates, the simple pole of the Riemann zeta function at
% \(1\), the gamma integral, Stirling's formula, and the elementary
% properties of holomorphic logarithms are used in their standard forms.
\section{Improved Error Term}\label{sec:improved-error}

We now complement Theorem~\ref{thm:main} with an upper bound having the
same threshold constant, apart from the endpoint.  For \(z\geq2\), put
\[
  \mathcal P(z):=\prod_{\substack{q<z\\q\ \mathrm{prime}}}q,
  \qquad
  \mathcal V(z):=\prod_{\substack{q<z\\q\ \mathrm{prime}}}
    \left(1-\frac1q\right),
\]
and define
\[
  \Psi_{\mu^2}(x,z)
  :=\#\{d\leq x:\mu^2(d)=1,\ P^+(d)<z\}.
\]
Thus \(\mathcal V(p)=V(p)\) when \(p\) is prime.

We use the following interval form of the fundamental lemma. 
\begin{lemma}
\label{lem:interval-fundamental}
Let \(I\) be a bounded real interval of length \(H\), and let \(z\geq2\)
and \(u\geq1\).  Then, uniformly as \(u\to\infty\),
\begin{align*}
  \#\{a\in I\cap\mathbb Z:(a,\mathcal P(z))=1\}
  ={}&H\mathcal V(z)
  \left(1+O\left(u^{-u(1+o(1))}\right)\right) \notag\\
  &\quad+O\left(\Psi_{\mu^2}(z^u,z)\right).
  \label{eq:interval-fundamental}
\end{align*}
Here the \(o(1)\) depends only on \(u\).
\end{lemma}
\begin{proof}
This is the case \(k=1\) of \cite[Lemma~1.4]{Gorodetsky2024}, which is deduced there from \cite[Theorem~6.12]{FriedlanderIwaniec2010}.
\end{proof}

The squarefree smooth-number term in
Lemma~\ref{lem:interval-fundamental} can be estimated directly by
Rankin's method.

\begin{lemma}\label{lem:squarefree-smooth-bound}
Let \(x\geq z\geq3\), put
\[
  U:=\frac{\log x}{\log z},
\]
and suppose that \(\e^2\leq U<z\).  Then
\begin{equation}\label{eq:squarefree-smooth-bound}
  \Psi_{\mu^2}(x,z)
  \leq
  x\exp\left\{
    -U\log U
    +O\left(\frac{U}{\log U}+\log\log z\right)
  \right\}.
\end{equation}
The implied constant is absolute.
\end{lemma}

\begin{proof}
Set
\[
  \eta:=\frac{\log U}{\log z},
  \qquad
  \sigma:=1-\eta.
\]
The hypotheses imply \(0<\sigma<1\).  Rankin's argument gives
\begin{align*}
  \Psi_{\mu^2}(x,z)
  \leq
  x^\sigma
  \sum_{\substack{d\geq1\\P^+(d)<z}}
  \frac{\mu^2(d)}{d^\sigma} =x^\sigma
  \prod_{q<z}\left(1+q^{-\sigma}\right).
\end{align*}
Since \(x=z^U\), we have
\[
  x^\sigma=x\exp\{-U\log U\}.
\]
Moreover,
\[
  \log\prod_{q<z}\left(1+q^{-\sigma}\right)
  \leq \sum_{q<z}q^{-1+\eta}.
\]
The standard Chebyshev bound \(\pi(t)\ll t/\log t\) and partial
summation give
\begin{align*}
  \sum_{q<z}q^{-1+\eta}\ll \frac{z^\eta}{\log z}
       +\int_2^z\frac{t^{-1+\eta}}{\log t}\dd t \ll \log\log z+\frac{U}{\log U}.
\end{align*}
For the last estimate, split the integral at \(t=e^{1/\eta}\) and, in the
second part, make the change of variables \(w=\eta \log t\).  Combining these
estimates proves \eqref{eq:squarefree-smooth-bound}.
\end{proof}

We next record the elementary estimate needed for the large-prime tail.
Recall the sets \(M_p(n)\) from
\eqref{eq:least-prime-decomposition}. 
With $I_p$ as defined in~\eqref{eq:Mp-rough-interval}, set
\begin{equation}\label{eq:Hp-length}
  H_p:=|I_p|
  =\frac{n-p-1}{p(p+1)}
  =\frac{n}{p(p+1)}-\frac1p.
\end{equation}

\begin{lemma}\label{lem:improved-large-prime-tail}
Uniformly for \(3\leq Y\leq\sqrt n\),
\[
  \sum_{\substack{p>Y\\p<n}}|M_p(n)|
  \ll \frac nY+\sqrt n.
\]
\end{lemma}

\begin{proof}
Suppose first that \(p>\sqrt n\).  If \(pa\in M_p(n)\), then
\(a<n/p<\sqrt n<p\).  By~\eqref{eq:Mp-rough-interval}, we have \(a> n/(p+1)\geq 1\) provided $p<n$. Therefore, \(P^-(a) < p\), a contradiction.
Thus \(M_p(n)=\varnothing\) for \(p>\sqrt n\).

For \(Y<p\leq\sqrt n\), equation \eqref{eq:Mp-rough-interval} gives
\[
  |M_p(n)|\leq H_p+1\ll\frac{n}{p^2}+1.
\]
Consequently,
\[
  \sum_{\substack{p>Y\\p<n}}|M_p(n)|
  \ll n\sum_{m>Y}\frac1{m^2}+\sqrt n
  \ll\frac nY+\sqrt n.
\]
\end{proof}

We are now ready to prove Theorem~\ref{thm:improved-error}.
\begin{proof}[Proof of Theorem~\ref{thm:improved-error}]
Recall from~\eqref{eq:least-prime-decomposition} that
\begin{align*}
    s_1(n) = 1+\sum_{p<n}|M_p(n)|.
\end{align*}
Put
\[
  Y:=\exp(2\sqrt{\log n\log\log n}).
\]
% For all sufficiently large \(n\), we have \(3<2\log n<Y<\sqrt n\).
We split the prime sum into the ranges $p\leq 2\log n$, $2\log n<p\leq Y$, and $p>Y$.

We first dispose of the primes \(p\leq2\log n\).  Inclusion-exclusion in \eqref{eq:Mp-rough-interval} gives
\begin{align*}
  |M_p(n)|
  =\sum_{d\mid\mathcal P(p)}
    \mu(d)\#\{a\in I_p\cap\mathbb Z:d\mid a\} =H_pV(p)+O\left(2^{\pi(p)}\right).
\end{align*}
The Chebyshev estimate \(\pi(x)\ll x/\log x\) therefore yields
\begin{equation}\label{eq:small-p-total}
  \sum_{p\leq2\log n}
  \left||M_p(n)|-H_pV(p)\right|
  \leq \exp\left\{O\left(\frac {\log n}{\log\log n}\right)\right\}.
\end{equation}

Now let \(2\log n<p\leq Y\), and set
\[
  u_p:=\frac{\log H_p}{\log p}.
\]
Uniformly in this range,
\begin{equation}\label{eq:up-uniform}
  \log H_p=\log n-2\log p+O(1),
  \qquad
  u_p=(1+o(1))\frac {\log n}{\log p},
\end{equation}
and hence
\[
  \min_{2\log n<p\leq Y}u_p
  \gg\sqrt{\frac {\log n}{\log\log n}}\longrightarrow\infty.
\]
We also have \(u_p<p\) for all sufficiently large \(n\).  Apply
Lemma~\ref{lem:interval-fundamental} to \(I_p\), with \(z=p\) and
\(u=u_p\), so that \(p^{u_p}=H_p\).  Lemma
\ref{lem:squarefree-smooth-bound}, together with \(V(p)\leq1\), then
gives, uniformly for \(2\log n<p\leq Y\),
\begin{equation}\label{eq:local-sieve-asymptotic}
  |M_p(n)|
  =H_pV(p)
  +O\left(
    H_p\exp\{-(1-o(1))u_p\log u_p\}
  \right).
\end{equation}
Indeed,
\[
  \frac{u_p/\log u_p+\log\log p}{u_p\log u_p}=o(1)
\]
uniformly throughout this range.

It remains to sum the errors in
\eqref{eq:local-sieve-asymptotic}.  Group the primes into blocks
\[
  \e^j\leq p<\e^{j+1},
\]
and put \(t=j+1\).  By \eqref{eq:up-uniform}, uniformly over the blocks
under consideration,
\[
  u_p\log u_p
  \geq(1-o(1))\frac {\log n}t\log\frac {\log n}t.
\]
Since \(H_p\leq n/p^2\) and
\[
  \sum_{\e^j\leq p<\e^{j+1}}\frac1{p^2}\ll\e^{-j},
\]
the contribution of the \(j\)-th block is at most
\[
  n\exp\left\{
    -j-(1-o(1))\frac {\log n}t\log\frac {\log n}t
  \right\}.
\]
Here \(t\leq2\sqrt{\log n\log\log n}+1\), and therefore
\[
  \log\frac {\log n}t
  \geq \left(\frac12-o(1)\right)\log\log n.
\]
The arithmetic--geometric mean inequality now gives
\begin{align*}
  t+(1-o(1))\frac {\log n}t\log\frac {\log n}t
  &\geq
  t+\left(\frac12-o(1)\right)\frac{\log n\log\log n}{t} \\
  &\geq (\sqrt2-o(1))\sqrt{\log n\log\log n}.
\end{align*}
There are only \(O(\sqrt{\log n\log\log n})\) blocks.  Consequently,
\begin{equation*}\label{eq:summed-sieve-error}
  \sum_{2\log n<p\leq Y}
  \left||M_p(n)|-H_pV(p)\right|
  \ll n\exp\{-(\sqrt2-o(1))\sqrt{\log n\log\log n}\}.
\end{equation*}
The contribution in \eqref{eq:small-p-total} is negligible compared to the right-hand side above,
so
\begin{equation}\label{eq:small-prime-total}
  \sum_{p\leq Y}|M_p(n)|
  =\sum_{p\leq Y}H_pV(p)
  +O\left(n\exp\{-(\sqrt2-o(1))\sqrt{\log n\log\log n}\}\right).
\end{equation}

The exact value of \(H_p\) in \eqref{eq:Hp-length} gives
\begin{align*}
  \sum_{p\leq Y}H_pV(p)
  &=n\sum_{p\leq Y}\frac{V(p)}{p(p+1)}
    -\sum_{p\leq Y}\frac{V(p)}p \\
  &=\Xi n+O\left(
    n\sum_{p>Y}\frac1{p^2}
    +\sum_{m\leq Y}\frac1m
  \right) \\
  &=\Xi n+O\left(\frac nY+\log Y\right).
\end{align*}
Substituting $Y$, we obtain
\begin{align}\label{eq: Main Term contri till Y}
    \sum_{p\leq Y}H_pV(p)=\Xi n+O\left(
    n\exp\left\{
      -(\sqrt2-o(1))\sqrt{\log n\log\log n}
    \right\}
  \right).
\end{align}
Furthermore, Lemma~\ref{lem:improved-large-prime-tail} gives
\[
  \sum_{\substack{p>Y\\p<n}}|M_p(n)|
  \ll\frac nY+\sqrt n\ll 
    n\exp\left\{
      -(\sqrt2-o(1))\sqrt{\log n\log\log n}
    \right\}
 .
\]
% Since \(Y=\e^{2B}\), all of
% \[
%   \frac nY,
%   \qquad \sqrt n,
%   \qquad \log Y,
%   \qquad \exp\left\{O\left(\frac {\log n}{\log\log n}\right)\right\}
% \]
% are absorbed by \(n\exp\{-(\sqrt2-o(1))\sqrt{\log n\log\log n}\}\).  
Finally,
\eqref{eq:small-prime-total} and~\eqref{eq: Main Term contri till Y} yield
\[
  s_1(n)
  =\Xi n+O\left(
    n\exp\left\{
      -(\sqrt2-o(1))\sqrt{\log n\log\log n}
    \right\}
  \right).
\]
\end{proof}

\begin{remark}
Together with Theorem~\ref{thm:main}, Theorem~\ref{thm:improved-error}
shows that \(\sqrt2\) is the sharp threshold constant away from the
endpoint.  The present argument does not decide whether the bound with
constant exactly \(\sqrt2\) holds.
\end{remark}

\begingroup
\raggedright

\endgroup


\begin{thebibliography}{99}

\bibitem{AbramowitzStegun1964}
M.~Abramowitz and I.~A. Stegun, eds.,
\emph{Handbook of Mathematical Functions with Formulas, Graphs, and
Mathematical Tables},
National Bureau of Standards Applied Mathematics Series, vol.~55,
U.S. Government Printing Office, Washington, DC, 1964.

\bibitem{BaraskarVukusic2026}
O.~Baraskar and I.~Vukusic,
\emph{Bounds for Sets of Remainders},
Journal of Integer Sequences \textbf{29} (2026), Article 26.1.4.
\url{https://cs.uwaterloo.ca/journals/JIS/VOL29/Vukusic/vukusic2.html}.

\bibitem{BaraskarGokhaleJainKiezun2026}
O.~Baraskar, P.~Gokhale, S.~Jain and A.~Kie\.zun,
\emph{Asymptotic Behavior of Iterated Sets of Remainders},
arXiv:2609.24412 [math.NT], 2026. \url{https://arxiv.org/abs/2609.24412}.

\bibitem{ChasePandey2022}
Z.~Chase and M.~Pandey,
\emph{On the Length of Pierce Expansions},
arXiv:2211.08374 [math.NT], 2022.
\url{https://arxiv.org/abs/2211.08374}.

\bibitem{ErdosShallit1991}
P.~Erd\H{o}s and J.~O. Shallit,
\emph{New Bounds on the Length of Finite Pierce and Engel Series},
Journal de Th\'eorie des Nombres de Bordeaux \textbf{3} (1991), no.~1,
43--53.
\url{https://doi.org/10.5802/jtnb.41}.

\bibitem{FriedlanderIwaniec2010}
J.~Friedlander and H.~Iwaniec,
\emph{Opera de Cribro},
American Mathematical Society Colloquium Publications, vol.~57,
American Mathematical Society, Providence, RI, 2010.

\bibitem{FGHM}
J.~B. Friedlander, A.~Granville, A.~Hildebrand and H.~Maier,
\emph{Oscillation theorems for primes in arithmetic progressions and for
sifting functions},
J. Amer. Math. Soc. \textbf{4} (1991), no.~1, 25--86.

\bibitem{Gorodetsky2024}
O.~Gorodetsky,
\emph{The Variance of Integers Without Small Prime Factors in Short
Intervals},
Mathematische Zeitschrift \textbf{308} (2024), no.~4, Paper No.~59.
\url{https://doi.org/10.1007/s00209-024-03601-w}.

\bibitem{GorodetskySmooth}
O.~Gorodetsky,
\emph{Smooth numbers and the Dickman $\rho$ function},
J. Anal. Math. \textbf{151} (2023), 139--169.

% \bibitem{Hildebrand1990}
% A.~Hildebrand,
% \emph{The Asymptotic Behavior of the Solutions of a Class of
% Differential-Difference Equations},
% Journal of the London Mathematical Society (2) \textbf{42} (1990),
% 11--31.
% \url{https://doi.org/10.1112/jlms/s2-42.1.11}.

\bibitem{HildebrandTenenbaum}
A.~Hildebrand and G.~Tenenbaum,
\emph{On integers free of large prime factors},
Trans. Amer. Math. Soc. \textbf{296} (1986), no.~1, 265--290.

% \bibitem{HildebrandTenenbaum1993}
% A.~Hildebrand and G.~Tenenbaum,
% \emph{On a Class of Differential-Difference Equations Arising in Number
% Theory},
% Journal d'Analyse Math\'ematique \textbf{61} (1993), 145--179.
% \url{https://doi.org/10.1007/BF02788841}.

% \bibitem{Mertens1874}
% F.~Mertens,
% \emph{Ein Beitrag zur analytischen Zahlentheorie},
% Journal f\"ur die reine und angewandte Mathematik \textbf{78} (1874),
% 46--62.
% \url{https://doi.org/10.1515/crll.1874.78.46}.

\bibitem{MontgomeryVaughan2007}
H.~L. Montgomery and R.~C. Vaughan,
\emph{Multiplicative Number Theory I: Classical Theory},
Cambridge Studies in Advanced Mathematics, vol.~97,
Cambridge University Press, Cambridge, 2007.
\url{https://doi.org/10.1017/CBO9780511618314}.
\bibitem{oeis}
OEIS Foundation Inc.,
\emph{The On-Line Encyclopedia of Integer Sequences},
\url{https://oeis.org},
2026.
% \bibitem{Olver1997}
% F.~W.~J. Olver,
% \emph{Asymptotics and Special Functions},
% A K Peters, Wellesley, MA, 1997.
% \url{https://doi.org/10.1201/9781439864548}.

% \bibitem{Weingartner2026}
% A.~Weingartner,
% \emph{Explicit Bounds for Buchstab's Function},
% arXiv:2607.21883 [math.NT], 2026.
% \url{https://arxiv.org/abs/2607.21883}.

\end{thebibliography}
\end{document}